\documentclass[12pt,reqno]{amsart}
\usepackage{mathrsfs}
\usepackage{cases}
\usepackage{epic}
\usepackage{amsfonts}
\usepackage{graphicx}
\usepackage{amsmath}
\usepackage{amssymb, upgreek}
\usepackage{bm}
\usepackage{latexsym,todonotes}
\usepackage{pdflscape}
\usepackage{color}
\usepackage{colordvi}
\usepackage{multicol}
\usepackage[normalem]{ulem}
\usepackage[mathscr]{euscript}

\input diagxy
\allowdisplaybreaks
\usepackage[linktocpage=true
]{hyperref}
\hypersetup{colorlinks,linkcolor=blue,urlcolor=cyan,citecolor=blue}

\usetikzlibrary{
	cd,
	shapes,
	arrows,
	positioning,
	decorations.markings,
	decorations.pathmorphing,
	circuits.logic.US,
	circuits.logic.IEC,
	fit,
	calc,
	plotmarks,
	matrix
}
	\input xy
	\xyoption{all}
	
	\newtheorem{innercustomthm}{{\bf Theorem}}
	\newenvironment{customthm}[1]
	{\renewcommand\theinnercustomthm{#1}\innercustomthm}
	{\endinnercustomthm}

	\newtheorem{theorem}{Theorem}[section]

	\newtheorem{corollary}[theorem]{Corollary}
	
	\newtheorem{definition}[theorem]{Definition}

	\newtheorem{lemma}[theorem]{Lemma}

	\newtheorem{proposition}[theorem]{Proposition}

	\numberwithin{equation}{section}
	
	\theoremstyle{remark}
	\newtheorem{remark}[theorem]{Remark}
\newcommand{\arxiv}[1]{\href{http://arxiv.org/abs/#1}{\tt arXiv:\nolinkurl{#1}}}

\def \hg {\widehat{\mathfrak{g}}}

\def \g {\mathfrak{g}}
\def \n{\mathfrak{n}}
\def \tpsi{\widetilde{\psi}}
\def \Z{\mathbb{Z}}

\def \Gr{\mathrm{Gr}}
\def \TH{\Theta}

\def \I{\mathbb{I}}
\def \N{\mathbb{N}}
\def \TH{\Theta}
\def \Sym{\mathrm{Sym}}
\def \bS{\mathbb{S}}
\def \ov{\overline}
\def \C{\mathbb{C}}
\def \bK{\mathbf{K}}
\def \Y{\mathbf{Y}^\imath} 
\def \cR{\mathcal{R}}
\def \A{\mathcal{A}}

\def \H{h}
\def \X{b}
\def \h{\hbar}
\def \oX{\ov{\X}}
\def \oH{\ov{\H}}

\def \B{\mathbf{B}}
\def \ad{\mathrm{ad}\,}

\def \L{\mathcal{L}}

\def \bTH { \boldsymbol{\Theta}}
\def \bH{\mathbf{H}}
\def \bDel{ \boldsymbol{\Delta}}
\def \bB{\mathbf{B}}
\def \bh{\mathbf{h}}
\def \bb{\mathbf{b}}

\def \ttb{\mathtt{b}}
\def \U{\mathbf U}
\def \Ui{\mathbf{U}^\imath}
\def \tUi{\mathbf{U}^\imath}
\def \tUiA{\mathbf{U}_{\mathcal{A}}^\imath }

\newcommand{\qbinom}[2]{\begin{bmatrix} #1\\#2 \end{bmatrix} }

\newcommand{\blue}[1]{{\color{blue}#1}}

\begin{document}

\title[Affine $\imath$quantum groups and twisted Yangians]{Affine $\imath$quantum groups and twisted Yangians in Drinfeld presentations}

\author{Kang Lu}
 \author{Weiqiang Wang}
 \address{Department of Mathematics, University of Virginia,
Charlottesville, VA 22903, USA}\email{kang.lu@virginia.edu, ww9c@virginia.edu}
\author{Weinan Zhang}
\address{Department of Mathematics and New Cornerstone
Science Laboratory, The University of Hong Kong, Hong Kong SAR, P.R.China}
\email{mathzwn@hku.hk}


\begin{abstract}
We formulate a family of algebras, twisted Yangians (of split type) in current generators and relations, via a degeneration of the Drinfeld presentation of affine $\imath$quantum groups (associated with split Satake diagrams). These new algebras admit PBW type bases and are shown to be a deformation of twisted current algebras; presentations for twisted current algebras are also provided. For type AI, it matches with the Drinfeld presentation of twisted Yangian obtained via Gauss decomposition. We conjecture that our split twisted Yangians are isomorphic to the corresponding ones in RTT presentation.  	
\end{abstract}

	\maketitle
		\setcounter{tocdepth}{1}
\tableofcontents

\section{Introduction}

\subsection{Background}

Drinfeld defined the Yangians in the $J$-presentation and subsequently gave a new current presentation exhibiting their structures as a deformation of current algebras; cf. \cite{Dr87, Dr88}. Yangians can also be defined using an RTT presentation; cf. \cite{MNO96, AACFR03, Mol07, JLM18, Wen18}. 

Associated with symmetric pairs of classical types, twisted Yangians in RTT presentations have been formulated and studied in depth in \cite{GR16, GRW17}, generalizing the original construction by Olshanski \cite{Ol92} (also cf. \cite{MNO96, Mol07}) for type AI and AII. Twisted Yangians can be viewed as coideal subalgebras of Yangians. A formulation of twisted Yangians in $J$-presentations has also been available \cite{Ma02, BR14}.  

Yangians and the Yang-Baxter equation arise in the
studies of exactly solvable two-dimensional statistical models and quantum integrable systems without boundaries (cf. \cite{Dr87} for extensive references). On the other hand, twisted Yangians and reflection equations arise from $(1+1)$-dimensional quantum field theories on a half-line and integrable systems with boundaries; see \cite{Ch84, Skl88}. Twisted Yangians have found some other physical applications; cf. e.g. \cite{MR11, MR12}.

A Drinfeld type current presentation for twisted Yangians has been missing for decades in the literature until very recently, when the authors \cite{KLWZ23} succeeded in constructing such a presentation for twisted Yangians of type AI through Gauss decomposition; a similar Gauss decomposition approach was used previously to construct Drinfeld presentations (or more generally parabolic presentations) from the RTT presentation of Yangians ${\mathbf Y}(\mathfrak{gl}_N)$ in \cite{BK05}; also cf. \cite{JLM18}. It remains highly nontrivial to generalize the Gauss decomposition approach further to obtain current presentations of twisted Yangians beyond type AI.

\subsection{Goal}
A degeneration approach relating affine quantum groups to Yangians was stated by Drinfeld \cite{Dr87} and then fully developed in \cite{GTL13} and differently in \cite{GM12}. In addition, Conner and Guay \cite{CG15} related via degeneration twisted quantum loop algebras to Olshanski's twisted Yangians in RTT presentations. Note that the relevant presentations and PBW basis properties for both affine quantum groups and (twisted) Yangians were known in the literature in the 20th century. 

The goal of this paper is to develop a degeneration approach to relate affine $\imath$quantum groups to twisted Yangians of arbitrary split types in Drinfeld type current presentations; these quantum algebras are associated to the split symmetric pairs (see Table~\ref{table:decomposition}). For us, formulating the current presentations of twisted Yangians and establishing their PBW bases are part of the new challenges. 

This can be viewed as a new extension of the $\imath$-program \cite{BW18}, which aims at generalizing various fundamental constructions from quantum groups to $\imath$quantum groups.

\subsection{Affine $\imath$quantum groups}

The $\imath$quantum groups arising from quantum symmetric pairs can be viewed as a vast generalization of Drinfeld-Jimbo quantum groups. The (affine) $\imath$quantum groups of split type were formulated in \cite{BB10} in connection to boundary affine Toda field theories, and they form a distinguished class of $\imath$quantum groups of Kac-Moody type \cite{Ko14}. 

The Drinfeld presentations for affine $\imath$quantum groups were obtained in \cite{LW21} for split ADE types and then in \cite{Z22} for split BCFG types. One of the most challenging aspects in such Drinfeld presentations is formulating the current Serre relations. Denote by $(c_{ij})_{i,j\in \I^0}$ the (finite) Cartan matrix with $\text{diag}(d_1, \ldots, d_n)$ as its symmetrizer. Explicit formulas for the current Serre relations for $c_{ij}=-1, -2$ were given {\em loc. cit.} and they are complicated because of the appearance of lower terms. 
In the case when $c_{ij}=-3$, it became too complicated to write down such Serre relations in general; instead reduced (equal-index) Serre relations were obtained and shown in \cite{Z22} to suffice as defining relations for Drinfeld presentations of affine $\imath$quantum groups. 

As the example for twisted Yangians of type AI already indicates, the current Serre relation for $c_{ij}=-1$ is already highly nontrivial (cf. \cite[(5.5)]{KLWZ23}). One expects that the current Serre relations for twisted Yangians get more complicated for $c_{ij}=-2$ or $-3$.

\subsection{Our approach}

A split $\imath$quantum group admits additional parameters $(\varsigma_i)_{i\in \I}$ beyond the quantum parameter $v$, which is a subtle point at the early stage of the project. To that end, having a Drinfeld presentation of twisted Yangians of type AI available from the Gauss decomposition approach in our previous work \cite{KLWZ23} helps us to make a suitable choice of the parameters. On the other hand, the current Serre relations discovered via the degeneration technique here were very helpful at the early stage of that work.

In Section~\ref{sec:degenerationloop}, we formulate the twisted loop algebras $\g[t,t^{-1}]^\omega$ and twisted current algebras $\g[u]^{\check\omega}$, as fixed point subalgebras of loop algebras and current algebras with respect to (extended) Chevalley involutions. As a warm-up preparation for the forthcoming sections, we explain how to degenerate the enveloping algebra $U(\g[t,t^{-1}]^\omega)$ to $U(\g[u]^{\check\omega})$. We also provide presentations for $U(\g[u]^{\check\omega})$ which are to be used in proving the PBW basis theorem for twisted Yangians in later sections. 

In Section~\ref{sec:iQG}, we review the Drinfeld type presentations of affine $\imath$quantum groups $\Ui$ from \cite{LW21, Z22}.


We formulate in Definition~\ref{def:YN} the twisted Yangians $\Y$ of split types A--F in Drinfeld presentations in terms of generators $\{ \H_{i,s},\X_{j,r} \mid s\in 2\N+1,r\in\N,i,j\in \mathbb{I}^0 \}$, with defining relation including explicit general current Serre relations. As shown in later sections, these relations are obtained by applying a suitable degeneration to the Drinfeld type defining relations of affine $\imath$quantum groups. More precisely, we construct a filtration on $\Ui$ with associated graded $\Gr_{\mathbf K} \tUi$ given in \eqref{GrK}, generalizing the construction in \cite{CG15}. Our main results are the following 2 theorems.

\begin{customthm} {\bf A} [Theorem~\ref{thm:qiso}]
\label{th:A}
There is an algebra isomorphism
\begin{align*}
\begin{split}
\Phi: & \Y \longrightarrow \Gr_{\bK}\tUi, 
\\
& \X_{i,r}  \mapsto \ov{\X}_{i,r,1},\quad
\H_{i,s} \mapsto \ov{\H}_{i,s,1}, 
\end{split}
\end{align*}
for $i\in \I^0,r\in \N, s\in 2\N+1$, where $\X_{i,r,1}, \H_{i,s,1} \in \tUi$ are defined in \eqref{def:HX} and $\ov{\X}_{i,r,1},\ov{\H}_{i,s,1}$ are their images in $\Gr_{\bK}\tUi$. 
\end{customthm}

Let $\cR^+$ denote the set of positive roots associated to a simple system $\{\alpha_i| i\in \I^0\}$ for $\g$. We form certain elements $\X_{\alpha,r} \in \Y$ by iterated commutators as in \eqref{def:bfalpha}. 

\begin{customthm} {\bf B}   [Theorem~\ref{thm:PBW}]
\label{th:B}
The ordered monomials of $\{\X_{\alpha,r},\H_{i,s} \mid \alpha\in \cR^+,i\in \I^0,r\in \N, s\in 2\N+1 \}$ form a basis of $\Y$.
\end{customthm}

Assume for now that $\Phi: \Y \rightarrow \Gr_{\mathbf K} \tUi$ is a homomorphism. Adapting arguments from \cite{GM12}, we establish a PBW basis for $\Gr_{\mathbf K} \tUi$. Then we show that a spanning set for $\Y$ is mapped to a basis for $\Gr_{\mathbf K} \tUi$, completing the proofs of Theorems~\ref{th:A} and \ref{th:B}.

It remains to show that $\Phi: \Y \rightarrow \Gr_{\mathbf K} \tUi$ is a homomorphism. To that end, we shall verify that all the defining relations for $\Y$ are satisfied by the images of the generators of $\Y$ under $\Phi$. All the relations in component-wise forms (except the current Serre relations \eqref{ty6} for $c_{ij}=-2$) are verified in Section~\ref{sec:verify}. 
In Section~\ref{sec:gen}, we reformulate all the defining relations for $\Y$ in generating function form and then verify them (including \eqref{ty6}) in $\Gr_{\mathbf K} \tUi$.

In light of Theorem~\ref{thm:Yreduced}, $\Y$ admits an equivalent presentation which only requires the simpler finite type Serre relations as follows (see \eqref{ty4}--\eqref{ty6} for general current Serre relations for $c_{ij}=-1,-2$).

\begin{customthm} {\bf C} \label{def:YNreduced}
\label{th:C}
The twisted Yangian $\Y$ of split type $(\I, \mathrm{Id})$ is isomorphic to the $\C[\h]$-algebra generated by $\H_{i,s},\X_{j,r}$, for $s\in 2\N+1,r\in\N,i,j\in \mathbb{I}^0$, subject to the relations: 
\begin{align*}
[\H_{i,s},\H_{j,r}] &=0,
\\
[\H_{i,1},\X_{j,r}]& =2  c_{ij}\X_{j,r+1},
\\
[\X_{j,r},\H_{i,s+2}]-[\X_{j,r+2},\H_{i,s}] &=-\h d_i c_{ij}( \X_{j,r+1}\H_{i,s} + \H_{i,s}\X_{j,r+1})+ \frac{\h^2 d_i^2 c^2_{ij}}{4}[\X_{j,r},\H_{i,s}] ,
\\
[\X_{i,s},\X_{j,r+1}]-[\X_{i,s+1},\X_{j,r}] &+\frac{\h d_i c_{ij}}{2}(\X_{j,r}\X_{i,s}+\X_{i,s}\X_{j,r})
=2\delta_{ij} (-1)^s \H_{i,r+s+1},
\end{align*}
and the additional finite type Serre relations: 
\begin{align*}
 [b_{i,0},b_{j,0}] 
&=0 \qquad\qquad\qquad\;\hskip 2.85cm (c_{ij}=0),
\\
 \big[b_{i,0},[b_{i,0},b_{j,0}]\big] 
&=- b_{j,0} \qquad\qquad\hskip 2.8cm (c_{ij}=-1),
\\
 \Big[b_{i,0},\big[b_{i,0},[b_{i,0},b_{j,0}]\big]\Big] 
 &= - 4 [b_{i,0},b_{j,0}] \hskip 3.35cm (c_{ij}=-2),
\\
%
\Big[b_{i,0},\Big[b_{i,0},\big[b_{i,0},[b_{i,0},b_{j,0}]\big]\Big]\Big]
&= -10 \big[b_{i,0},[b_{i, 0},b_{j,0}]\big] -9 b_{j,0}\qquad (c_{ij}=-3).
\end{align*}
\end{customthm}

We take the liberty of including in the above presentation the twisted Yangians of split type for type $G_2$, which require the finite type Serre relation for $c_{ij}=-3$ (see Section~\ref{sec:G2}). For type $G_2$, we use this presentation as the definition of the twisted Yangian since we do not have an explicit general current Serre relations for $c_{ij}=-3$. Nevertheless, we can still establish a PBW basis for the twisted Yangian of type $G_2$ and show that it is indeed a flat deformation of the enveloping algebra of the twisted current algebra; in other words, Theorems~\ref{th:A} and \ref{th:B} remain valid for type $G_2$. 

For type AI, the Drinfeld presentation of twisted Yangians obtained via degeneration coincides with the one obtained via Gauss decomposition approach in our previous work \cite{KLWZ23}; and hence, the twisted Yangian of type AI here is isomorphic to the one in RTT presentation \cite{Ol92}. We conjecture that the general twisted Yangians of split type in Drinfeld type current presentation given in this paper are isomorphic to the corresponding twisted Yangians in RTT presentations given in \cite{GR16}. Hopefully such an isomorphism can be established by further developing the Gauss decomposition approach. 

A Drinfeld presentation for twisted Yangians of quasi-split type will be developed by two of the authors \cite{LZ24} using approaches of the Gauss decomposition and the degeneration (cf. \cite{LWZ24}).

\section{Degeneration from loop and twisted loop algebras}
\label{sec:degenerationloop}

In this section, by ``twisted" we refer to fixed point subalgebras of current algebras and loop algebras under Chevalley involutions. We show how to obtain twisted current algebras from degeneration of twisted loop algebras. In addition, we give two presentations of a twisted current algebra. 

\subsection{From loop algebras to current algebras} 
 \label{subsec:loop}
 
Let $\I=\{0,1,\ldots,n\}$ and $\I^0=\{1,\ldots,n\}$. Let $\hg$ be an untwisted affine Kac-Moody algebra with Cartan matrix $C=(c_{ij})_{i,j\in\mathbb{I}}$. Let $\mathfrak{g}$ be the simple Lie algebra with Cartan matrix $C^0=(c_{ij})_{i,j\in\mathbb{I}^0}$. Let $D=\mathrm{diag}(d_0,\ldots,d_n)$, where $d_i\in \N$ and $\gcd(d_0,\ldots, d_n) =1$, be the diagonal matrix such that $DC$ is symmetric. 

Let $\cR$ denote the root system of $\g$, $\{\alpha_i| i\in \I^0\}$ denote a simple system of $\cR$, and  $\cR^+$ denote the corresponding set of positive roots.
Let $\alpha_{\max}$ be the highest root in $\cR^+$. Set $\alpha_0 =\delta-\alpha_{\max}$, where $\delta$ denote the basic imaginary root of $\hg$. Then $\{\alpha_i| i\in \mathbb{I}\}$ forms a set of simple roots for $\hg$. Recall that $\hg$ (with the degree operator omitted) has a loop algebra realization which fits into a short exact sequence of Lie algebras
\begin{equation}
\label{ses}
0 \longrightarrow \C c \longrightarrow \hg \longrightarrow \mathfrak{g}[t,t^{-1}]\longrightarrow 0.
\end{equation}

The (enveloping algebra of) current algebra $\g[u]$ can be obtained from the (enveloping algebra of) loop algebra $\g[t,t^{-1}]$ via a degeneration as follows (cf. \cite{GM12} for a quantum counterpart). Write $g_k:=gt^k$ for any $g\in \g,k\in \Z$.
Define $g_{r,k}\in \g[t,t^{-1}]$ for $g\in \g, k\in \Z, r\geq 0$ by
\begin{align}
\label{xrk}
g_{r,k}:=\sum_{s=0}^r (-1)^{r-s} \binom{r}{s} g_{s+k}=g(t-1)^r t^k.
\end{align}

For $r\geq 0$, set $\kappa_r$ to be the ideal of $ U\big(\g[t,t^{-1}]\big)$ generated by elements $(g_1)_{r_1,k_1} \cdots (g_a)_{r_a,k_a}$ with $g_i\in \g, k_i\in \Z, 1\leq i\leq a, r_1+\ldots + r_a \geq r, a\geq0$. Then $U\big(\g[t,t^{-1}]\big)$ admits a decreasing filtration 
\begin{align}  \label{filtration}
U\big(\g[t,t^{-1}]\big) =\kappa_0 \supset \kappa_1 \supset \cdots \supset \kappa_m \supset \cdots, 
\end{align}
and we denote the associated graded algebra by 
\[
\Gr_{\kappa}\g[t,t^{-1}]=\bigoplus_{r\geq 0} \kappa_r/ \kappa_{r+1}. 
\]
It is clear from the definition that
$g_{r,k}\in \kappa_r/ \kappa_{r+1}.$
Write $\ov{g_{r,k}}$ for the image of $g_{r,k}$ in  $\kappa_r / \kappa_{r+1}$. Since $g_{r+1,k}=g_{r,k+1}-g_{r,k}$, we have that 
\begin{align}\label{climit}
\ov{g_{r,k}}=\overline{g_{r,k+1}} \in  \Gr_{\kappa}\g[t,t^{-1}], \text{ for }k\in \Z.
\end{align}

\begin{proposition}\label{prop:iso}
There is an algebra isomorphism $\Phi: U(\g[u]) \stackrel{\cong}{\longrightarrow} \Gr_{\kappa}\g[t,t^{-1}],$ which maps $g u^r\mapsto \ov{g_{r,0}}$, for $g\in \g,r\geq 0$.  
\end{proposition}

\begin{proof}
For $z,y\in \g, k,l\in \Z$, we have by definition \eqref{xrk} that 
\begin{align} \label{xy}
[z_{r,k},y_{s,l}]=[zy]_{r+s,k+l}.
\end{align}
It follows from \eqref{xy} that  $[\ov{z_{r,0}},\ov{y_{s,0}}]=\ov{[zy]_{r+s,0}}$, and hence sending $g u^r$ to $\ov{g_{r,0}}$ defines an algebra homomorphism $\Phi: U(\g[u])\rightarrow \Gr_{\kappa}\g[t,t^{-1}]$.

We show that $\Phi$ is an isomorphism. By the PBW theorem, the ordered monomials of $g_{r,0}$ (for $r\geq 0$, where $g$ runs over in a fixed basis of $\g$) are linearly independent in $U\big(\g[t,t^{-1}]\big)$; such monomials lie in $\kappa_m$ if and only if the sum of indices $r$ is at least $m$. Hence, the ordered monomials of $\ov{g_{r,0}}$ with the sum of indices $r$ equal to $m$ are linearly independent in $\kappa_m / \kappa_{m+1} $; such monomials also span $\kappa_m / \kappa_{m+1} $ by \eqref{climit}. Therefore, we obtained a basis of $\Gr_{\kappa}\g[t,t^{-1}]$, and it is also clear that such a basis is the $\Phi$-image of a PBW basis for $U(\g[u])$.
\end{proof}

\begin{remark}
Proposition \ref{prop:iso} and its quantum variants are well known; cf. \cite{GM12, GTL13, Wen22}. A conceptual way of seeing this is as follows (as kindly explained by a referee).
Let $\tau_{-1}$ be the shift automorphism of $\g[u]$ (and hence of $U(\g[u])$) define by sending $xu^k\mapsto x(u-1)^k$. Let $\Gamma$ be the composition of $\tau_{-1}\colon U(\g[u])\rightarrow U(\g[u])$ and the embedding $U(\g[u])\hookrightarrow U(\g[t,t^{-1}])$ sending $u\mapsto t$. By definition, $\Gamma$ sends $gu^r$ to $g_{r,0}$ for $r\ge 0$, and thus the map $\Phi$ in Proposition~\ref{prop:iso} is the associated graded map $\text{gr}(\Gamma)$. The inverse of $\text{gr}(\Gamma)$ can be constructed directly or follows from an exponential map construction in \cite{GTL13}. A version of $\Gamma$ for Yangians can be found in \cite[(5.5)]{Wen22}.
\end{remark}

\subsection{Twisted loop algebras}
 \label{sec:climit}

Let $\{e_i,f_i,h_i|i\in \I^0\}$ be a set of Chevalley generators of $\g$, and in particular, $\{e_i,f_i,h_i\}$ forms a standard $\mathfrak{sl}_2$-triple, for each $i\in \I^0$. Denote by 
\[
\omega_0: \g \longrightarrow \g, \qquad f_i\mapsto -e_i, \; e_i\mapsto -f_i, \; h_i\mapsto -h_i,
\]
the Chevalley involution on $\g$. Denote by $\g^{\omega_0}$ the $\omega_0$-fixed point subalgebra of $\g$. Then $(\g, \g^{\omega_0})$ forms a split symmetric pair. Here is a complete list of split symmetric pairs (cf. \cite{OV90}).

\begin{table}[h]
\caption{Split symmetric pairs $(\g, \g^{\omega_0})$}
\label{table:decomposition}
\begin{tabular}{| c || c | c | c | c |c |  }
\hline
$\g$ & $\mathfrak{sl}_{n+1}$  & $\mathfrak{so}_{2n+1}$ & $\mathfrak{sp}_{2n}$  & $\mathfrak{so}_{2n}$  &
\\
\hline
$\g^{\omega_0}$ & $\mathfrak{so}_{n+1}$   &$\mathfrak{so}_{n} \oplus \mathfrak{so}_{n+1}$ &$\mathfrak{gl}_{n}$ &$\mathfrak{so}_{n} \oplus \mathfrak{so}_{n}$  &
\\
\hline\hline
$\g$  & $E_6$ & $E_7$ &$E_8$ &  $F_4$ &  $G_2$  
\\
\hline
 $\g^{\omega_0}$ & $\mathfrak{sp}_{8}$ &$\mathfrak{sl}_{8}$ &$\mathfrak{so}_{16}$ & $\mathfrak{sl}_{2} \oplus \mathfrak{sp}_{6}$   & $\mathfrak{so}_{3} \oplus \mathfrak{so}_{3}$   
\\
\hline
\end{tabular}
\end{table}

The involution $\omega_0$ extends to an involution $\omega$ on the loop algebra $\g[t,t^{-1}]$ by
\begin{align} \label{omegat}
 \begin{split}
\omega: \g[t,t^{-1}] & \longrightarrow \g[t,t^{-1}], \\
gt^k & \mapsto \omega_0(g) t^{-k}.
 \end{split}
\end{align}
The involution $\omega$ extends to an involution (also  denoted $\omega$) on $\hg$ in \eqref{ses} by sending $c\mapsto -c$. 

\begin{lemma}
\begin{enumerate}
    \item 
The fixed point subalgebra $\g[t,t^{-1}]^\omega$ of $\g[t,t^{-1}]$ is generated by
\begin{align}\label{def:loop}
\theta_{i,m} :=h_i t^m -h_i t^{-m}, \qquad b_{i,k}:=f_i t^{-k} - e_i t^{k},\qquad \forall k\in \Z,m \ge 1.
\end{align}
\item
The fixed point subalgebra $\g[t,t^{-1}]^\omega$ coincides with the fixed point subalgebra $\hg^\omega$ of $\hg$.
\end{enumerate}
(The algebra $\g[t,t^{-1}]^\omega=\hg^\omega$ will be referred to as a twisted loop algebra.)
\label{lem:loop}
\end{lemma}

\begin{proof}
(1). For each $\beta\in \cR^+$, we fix $i_1,\ldots, i_k\in \I^0$ such that $f_{\beta}:=\big[f_{i_1},[f_{i_2},\cdots[f_{i_{k-1}},f_{i_k}]\cdots]\big]$ is a nonzero root vector in $\g_{-\beta}$. Set $e_\beta=-\omega_0(f_\beta)\in \g_{\beta}$. Then $\{f_\beta t^k, h_i t^k,e_\beta t^k|k\in \Z, \beta \in \cR^+,i\in \I^0\}$ is a basis for $\g[t,t^{-1}]$. It follows that 
\begin{align}
\theta_{i,m} =h_i t^m -h_i t^{-m}, \qquad b_{\beta,k}:=f_\beta t^{-k} - e_\beta t^{k},\qquad
(m\geq1,k\in \Z,i\in \I^0,\beta\in \cR^+),
\end{align}
form a basis for $\g[t,t^{-1}]^\omega$. 

It suffices to show that $b_{\beta,k}$ lies in the subalgebra generated by $\theta_{i,m}, b_{i,k},k\in \Z, m\geq1$, by induction with respect to the standard partial order on $\cR^+$. The standard partial order on $\cR^+$ here is understood as: $\beta_1\leq \beta_2$ if and only if $\beta_2-\beta_1$ is a nonnegative linear combination of $\alpha_i$ for $i\in\I^0$.

If $\beta=\alpha_i$ for some $i\in \I^0$, then $b_{\beta,k}=b_{i,k}$ as desired. Otherwise, we can write $\beta=\alpha_{i_1}+\beta'$ for some $i_1\in \I^0,\beta'\in\cR^+$ such that $f_\beta$ is a nonzero scalar multiple of $[f_{i_1},f_{\beta'}]$, where $\beta'<\beta$. Write $f_\beta=\kappa[f_{i_1},f_{\beta'}]$ for some $\kappa \in\C^*$. Then the element
\begin{align}
b_{\beta,k}-\kappa  [b_{i_1,0},b_{\beta',k}] =  \kappa  [e_{i_1}, f_{\beta'}]t^{-k} -\kappa [e_{\beta'}, f_{i_1}]t^{k}
\end{align}
is a multiple of $b_{\beta'-\alpha_{i_1},k}$ (if $\beta'-\alpha_{i_1}$ is not a root, we simply set $b_{\beta'-\alpha_{i_1},k}=0$). Since $\beta'-\alpha_{i_1}<\beta'<\beta$, by induction hypothesis, both $b_{\beta'-\alpha_{i_1},k}$ and $b_{\beta',k}$ lie in the subalgebra generated by $\theta_{i,m}, b_{i,k},$ for $k\in \Z, m\geq1$. Therefore, $b_{\beta,k}$ also lies in this subalgebra as desired.

(2). Clearly, $\{f_\beta t^k, h_i t^k,e_\beta t^k,c|k\in \Z, \alpha\in \cR^+,i\in \I^0\}$ is a basis for $\hg$. Since $\omega(c)= -c$, this basis of $\hg$ implies that $\{\theta_{i,m},b_{\beta,k}|m\geq1,k\in \Z,i\in \I^0,\beta\in \cR^+\}$ is a basis for $\hg^\omega$ and the central extension when restricted to $\hg^\omega$ is trivial. Hence, $\hg^\omega$ and $\g[t,t^{-1}]^\omega$ are naturally identified. 
\end{proof}

\subsection{Filtration on twisted loop algebras}

We extend the involution $\omega_0$ on $\g$ to an involution $\check\omega$ on the current algebra $\g[u]$ by  
\begin{align}  \label{omegau]}
 \begin{split}
   \check\omega : \g[u] & \longrightarrow \g[u], 
   \\
    gu^r & \mapsto  (-1)^r \cdot \omega_0(g) u^r. 
 \end{split}
\end{align}
We call $\g[u]^{\check\omega}$ a twisted current algebra.

We shall formulate a degeneration from $U(\g[t,t^{-1}]^\omega)$ to $U(\g[u]^{\check\omega})$. A quantum generalization will be given in Section~\ref{sec:main}. 

For $m\geq 0$, define $\kappa^\imath_m$ to be the ideal $\kappa_m \cap U(\g[t,t^{-1}]^\omega)$ of $U(\g[t,t^{-1}]^\omega)$. 
Then $U\big(\g[t,t^{-1}]^\omega\big)$ admits a decreasing filtration 
\begin{align}  \label{filtration2}
U\big(\g[t,t^{-1}]^\omega\big) =\kappa_0^\imath \supset \kappa_1^\imath \supset \cdots \supset \kappa_m^\imath \supset \cdots. 
\end{align}
Since the two filtrations \eqref{filtration} and \eqref{filtration2} are compatible, i.e., $\kappa^\imath_m=\kappa_m \cap U(\g[t,t^{-1}]^\omega)$, the associated graded of \eqref{filtration2}, $\Gr_{\kappa^\imath}\g[t,t^{-1}]^\omega:=\bigoplus_{m\geq 0} \kappa_m^\imath / \kappa_{m+1}^\imath,$ can be regarded naturally as a subalgebra of $\Gr_{\kappa}\g[t,t^{-1}]$:
\begin{align}  \label{GrGr}
\Gr_{\kappa^\imath}\g[t,t^{-1}]^\omega  \subset \Gr_{\kappa}\g[t,t^{-1}].
\end{align}

Define, for $r\geq 0,l\geq 1$ and $k\in \Z$,
\begin{align}\label{def:cgr}
\theta_{i,r,l}=\sum_{s=0}^r (-1)^{r-s} \binom{r}{s} \theta_{i,s+l},\qquad \beta_{i,r,k}=\sum_{s=0}^r (-1)^{r-s} \binom{r}{s} b_{i,s+k}.
\end{align}
We denote by $\overline\theta_{i,r,l}$ and $\overline \beta_{i,r,k}$ their images in $\Gr_{\kappa^\imath}\g[t,t^{-1}]^\omega$. 
Viewed as elements in $U(\g[t,t^{-1}])$, the following identities hold:  
\begin{align}
\label{climit2}
\begin{split}
\theta_{i,r,k} &=h_i (t-1)^r t^k -h_i (1-t)^r t^{-r-k},
\\
\beta_{i,r,k} &=f_i(1-t)^r t^{-r-k}-e_i (t-1)^r t^k.
\end{split}
\end{align}

Recall the isomorphism $\Phi: U(\g[u]) \rightarrow \Gr_{\kappa}\g[t,t^{-1}]$ from Proposition~\ref{prop:iso} and keep in mind \eqref{GrGr}.

\begin{proposition}
  \label{prop:climit}
The following statements hold:
\begin{itemize}
\item[(1)] For $r$ even, we have $\theta_{i,r,k}\in \kappa_{r+1}^\imath$ and thus $\ov{\theta}_{i,r,k}=0$.
\item[(2)] For $r$ odd, we have $\theta_{i,r,k}\in \kappa_{r}^\imath \setminus \kappa_{r+1}^\imath$ and $\Phi^{-1}(\ov{\theta}_{i,r,k}) =2 h_i u^{r}$.
\item[(3)] We have  $\beta_{i,r,k}\in \kappa_r^\imath\setminus \kappa_{r+1}^\imath$ and $\Phi^{-1}(\ov{\beta}_{i,r,k})=\big( (-1)^r f_i -e_i \big)u^r$.
\end{itemize}
In particular, the preimages of $\ov{\theta}_{i,r,k},\ov{\beta}_{i,m,k}$ (for $r,m\in \N$ with $r$ odd) under $\Phi$ lie in the fixed point subalgebra $\g[u]^{\check\omega}$.
\end{proposition}

\begin{proof}
Follows from \eqref{climit2} and the definition of $\Phi$ from Proposition~\ref{prop:iso}.
\end{proof}

\begin{corollary}
The following identities hold  in $\Gr_{\kappa^\imath}\g[t,t^{-1}]^\omega$: 
\[
\ov{\theta}_{i,r,k}=\ov{\theta}_{i,r,k+1},
\qquad 
\ov{\beta}_{i,m,k}=\ov{\beta}_{i,m,k+1},
\]
for $k\in \Z, r,m\in \N$ with $r$ odd.
\end{corollary}

\begin{proof}
Follows by $\theta_{i,r+1,k}=\theta_{i,r,k+1}-\theta_{i,r,k}$, $\beta_{i,m+1,k}=\beta_{i,m,k+1}-\beta_{i,m,k}$ and Proposition~\ref{prop:climit}.
\end{proof}

\subsection{A presentation for $U\big(\g[u]^{\check\omega}\big)$}

Throughout $\Sym_{k_1,k_2,k_3}$ denotes the symmetrization among indices $\{k_1,k_2,k_3\}$ (i.e., a sum of $3!$ terms), and $\text{Cyc}_{k_1,k_2,k_3}$ denotes a sum over the 3 cyclic permutations of indices $(k_1,k_2,k_3).$ Notations $\Sym_{k_1,k_2}$, $\Sym_{w_1,w_2}$ and so on are understood similarly. Recall the Cartan matrix $C^0=(c_{ij})_{i,j\in\mathbb{I}^0}$ for $\g$. 

\begin{proposition}
  \label{prop:crel}
The algebra $U(\g[u]^{\check\omega})$ is generated by $t_{i,r},x_{i,m}$ for $i\in \I^0,r,m\geq 0$ with $r$ odd, subject to the relations \eqref{eq:classical1}--\eqref{eq:classical3}, for $i,j\in\I^0, m,k, r,s\in \Z_{\ge 0}$ with $r,s$ odd:
\begin{align}
\label{eq:classical1}
[t_{i,r},t_{j,s}] &=0,
\\
\label{eq:classical2}
[t_{i,r},x_{j,m}] &=c_{ij} x_{j,m+r},
\\
\label{eq:classical3}
[x_{i,k},x_{j,m+1}]-[x_{i,k+1},x_{j,m}] &=\delta_{ij}2 \big((-1)^k+(-1)^m\big)t_{i,k+m+1},
\end{align}
and the Serre relations \eqref{eq:classical4'}--\eqref{eq:classical6}:  
\begin{align}\label{eq:classical4'}
[x_{i,k},x_{j,m}] &=0, \qquad\qquad\qquad\qquad \text{for } c_{ij}=0;
\\
\begin{split}
\Sym_{k_1,k_2}\big[x_{i,k_1},[x_{i,k_2},x_{j,m}]\big] 
&= \big((-1)^{k_1+1}+(-1)^{k_2+1}\big)x_{j,k_1+k_2+m}, \\
&\qquad\qquad\qquad\qquad\qquad \text{ for } c_{ij}=-1;
\label{eq:classical4}
\end{split}
\\
\begin{split}
\label{eq:classical5}
\Sym_{k_1,k_2,k_3}\Big[x_{i,k_1},\big[x_{i,k_2},[x_{i,k_3},x_{j,m}]\big]\Big] 
&= 4\Sym_{k_1,k_2,k_3}  (-1)^{k_1+1} [x_{i,k_1+k_2+k_3},x_{j,m}], 
\\
&\qquad\qquad\qquad\qquad\qquad \text{ for } c_{ij}=-2;
\end{split}
\end{align}
and
\begin{align}\label{eq:classical6}
\begin{split}
&\Sym_{k_1,k_2,k_3,k_4}\bigg[x_{i,k_1},\Big[x_{i,k_2},\big[x_{i,k_3},[x_{i,k_4},x_{j,m}]\big]\Big]\bigg]
\\
&= 10\Sym_{k_1,k_2,k_3,k_4}  (-1)^{k_2+1} \big[x_{i,k_1},[x_{i, k_2+k_3+k_4},x_{j,m}]\big]
\\
&\quad -9 \Sym_{k_1,k_2,k_3,k_4} (-1)^{k_1+k_2} x_{j,k_1+k_2+k_3+k_4+m},
\quad \text{ for } c_{ij}=-3.
\end{split}
\end{align}
\end{proposition}

\begin{proof}
Let $\L$ be the algebra defined with generators $t_{i,r},x_{i,m}$ (for $r$ odd) and defining relations \eqref{eq:classical1}--\eqref{eq:classical6}. There is a homomorphism  
\begin{align} \label{eq:iso}
\begin{split}
    \rho: \L & \longrightarrow U(\g[u]^{\check\omega})
    \\
t_{i,r} & \mapsto h_i u^{r}, \quad x_{i,m}\mapsto \big( (-1)^m f_i - e_i \big)u^m. 
\end{split}
\end{align}
Indeed it is straightforward though tedious to check that the relations \eqref{eq:classical1}--\eqref{eq:classical6} are satisfied for the images in \eqref{eq:iso} in $U(\g[u]^{\check\omega})$.

For each $\alpha\in \cR^+$, we fix $i_1,\ldots, i_k\in \I^0$ such that $f_{\alpha}:=\big[f_{i_1},[f_{i_2},\cdots[f_{i_{k-1}},f_{i_k}]\cdots]\big]$ is a nonzero root vector in $\g_{-\alpha}$. Define, for $m\geq0$,
\begin{align}\label{def:xam}
x_{\alpha,m}=\Big[x_{i_1,0},\big[x_{i_2,0},\cdots[x_{i_{k-1},0},x_{i_k,m}]\cdots\big]\Big].
\end{align}
We claim that the $\rho$-images of $t_{i,r},x_{\alpha,m}$ form a basis of $\g[u]^{\check\omega}$. Set $e_\alpha=-\omega_0(f_\alpha)$. By definition, we have
\begin{align*}
\rho(x_{\alpha,n}) \in \big( (-1)^m f_\alpha - e_\alpha \big)u^m + \sum_{\alpha': \text{ht}(\alpha')<\text{ht}(\alpha)} \C \big( (-1)^m f_{\alpha'} - e_{\alpha'} \big)u^m,
\end{align*}
where $\text{ht}(\alpha)$ denotes the height of $\alpha$. Since $\{( (-1)^m f_\alpha - e_\alpha )u^m, h_i u^{2m+1} \mid i\in \I^0,\alpha\in \cR^+,m\geq 0 \}$ clearly form a basis of $\g[u]^{\check\omega}$, the $\rho$-images of $t_{i,r},x_{\alpha,m}$ (for $r$ odd) also form a basis of $\g[u]^{\check\omega}$. By the PBW theorem, the monomials in $\rho$-images of $t_{i,r},x_{\alpha,m}$ with respect to any fixed order form a basis of $U(\g[u]^{\check\omega})$.

It remains to show that the monomials of $t_{i,r},x_{\alpha,m}$ $(i\in \I^0, \alpha \in \cR^+, r,m \in \Z_{\ge 0}$ with $r$ odd) with respect to any fixed order form a spanning set of $\L$; this implies that $\rho$ is an isomorphism and hence proves the proposition.

Define a filtration on $\L$ by setting $\deg x_{i,m}=m+1,\deg t_{i,r}=r+1$. Denote by $\ov{\L}$ the associated graded algebra and denote by $\breve{t}_{i,r},\breve{x}_{j,m}$, $\breve{x}_{\alpha,m}$ the images of $t_{i,r},x_{j,m}, x_{\alpha,m}$ in $\ov{\L}$. Then we have that 
$[\breve{t}_{i,r},\breve{x}_{j,m}]=0$ in $\ov{\L}$ by \eqref{eq:classical2}, and hence $\ov{\L}$ is isomorphic to (a quotient of) $\C[\breve{t}_{i,r}|i\in \I^0,r\in 2\N+1]\otimes \ov{\L}^+$, where $\ov{\L}^+$ is the subalgebra of $\ov{\L}$ generated by all $\breve{x}_{j,m}$ and it satisfies the following relations (which are implied by \eqref{eq:classical3}--\eqref{eq:classical6}):
\begin{align*}
[\breve{x}_{i,s},\breve{x}_{j,r+1}]-[\breve{x}_{i,s+1},\breve{x}_{j,r}]&=0,
\\
[\breve{x}_{i,k},\breve{x}_{j,r}]&=0 \qquad (c_{ij}=0),
\\
\Sym_{k_1,k_2}\big[\breve{x}_{i,k_1},[\breve{x}_{i,k_2},\breve{x}_{j,r}]\big] &=0\qquad (c_{ij}=-1),
\\
\Sym_{k_1,k_2,k_3}\Big[\breve{x}_{i,k_1},\big[\breve{x}_{i,k_2},[\breve{x}_{i,k_3},\breve{x}_{j,r}]\big]\Big] & =0\qquad (c_{ij}=-2),
\\
\Sym_{k_1,k_2,k_3,k_4}\bigg[\breve{x}_{i,k_1},\Big[\breve{x}_{i,k_2},\big[\breve{x}_{i,k_3},[\breve{x}_{i,k_4},\breve{x}_{j,r}]\big]\Big]\bigg]&=0\qquad (c_{ij}=-3).
\end{align*}
Hence, $\ov{\L}^+$ is identified with (a quotient of) the current algebra $U(\n^-[u])$ where $\n^- =\oplus_{\alpha \in \cR^+}\g_{-\alpha}$. By the PBW theorem for the current algebra, the ordered monomials of $\breve{x}_{\alpha,m}$ form a spanning set of  $\ov{\L}^+$. 
\end{proof}

\subsection{From twisted loop algebras to twisted current algebras}\label{dege:notwisted}

Following the definition of $x_{\alpha,m}$ and the notation preceding \eqref{def:xam}, we define $\beta_{\alpha,m,l}$ for $\alpha\in \cR^+, m\geq 0, l\in \Z$ as
\begin{align} \label{def:betam}
\beta_{\alpha,m,l}=\Big[\beta_{i_1,0,0},\big[\beta_{i_2,0,0},\cdots[\beta_{i_{k-1},0,0},\beta_{i_k,m,l}]\cdots\big]\Big].
\end{align}

\begin{proposition}
There is an algebra isomorphism $\Phi_\imath: U(\g[u]^{\check\omega})\rightarrow\Gr_{\kappa^\imath}\g[t,t^{-1}]^\omega$ such that
\begin{align}
2t_{i,r} \mapsto \ov{\theta}_{i,r,1},\qquad x_{i,m}\mapsto\ov{\beta}_{i,m,1} \qquad (i\in \I^0, r,m\geq 0 \text{ with } r \text{ odd}).
\end{align}
Moreover, we have the following commutative diagram:
\[
\xymatrix{ U(\g[u]^{\check\omega}) \ar[r]^{\Phi_\imath} \ar[d]^{}&  \Gr_{\kappa^\imath}\g[t,t^{-1}]^\omega \ar[d]^{}\\
U(\g[u]) \ar[r]^{\Phi} & \Gr_{\kappa}\g[t,t^{-1}]}
 \]
where the vertical arrows are natural inclusions, cf. \eqref{GrGr}.
\end{proposition}

\begin{proof}
By Proposition~\ref{prop:climit}, $\Phi$ maps $2t_{i,r},x_{i,m}$ to $\ov{\theta}_{i,r,1},\ov{\beta}_{i,m,1}\in \Gr_{\kappa^\imath}\g[t,t^{-1}]^\omega$, respectively. By Proposition~\ref{prop:crel}, $t_{i,r},x_{i,m}$ generates $U(\g[u]^{\check\omega})$ and hence $\Phi$ restricts to a map $\Phi_\imath:U(\g[u]^{\check\omega}) \rightarrow \Gr_{\kappa^\imath}\g[t,t^{-1}]^\omega$. Since $\Phi$ is an isomorphism, $\Phi_\imath$ is injective.

It remains to show the surjectivity for $\Phi_\imath$. It suffices to show that the set $\{\ov{\theta}_{i,r,1},\ov{\beta}_{i,m,1}| i\in \I^0,m\geq 0,r\in 2\N+1\}$ generates $\Gr_{\kappa^\imath}\g[t,t^{-1}]^\omega$. 
By the proof of Lemma~\ref{lem:loop}, the monomials of $\theta_{i,r},b_{\alpha,k}, r\geq 1,k\in \Z,i\in \I^0,\alpha\in \cR^+$ for some fixed order form a PBW basis for $U(\g[t,t^{-1}]^\omega)$.

We define an $\N$-filtration on $U(\g[t,t^{-1}]^\omega)$ by setting $\deg \theta_{i,r}=r, \deg b_{i,k}=\deg b_{i,-k}=k$ for $r>0,k\geq 0$. Then $\deg b_{\alpha,k}=\deg b_{\alpha,-k}=k$ for $\alpha\in \cR^+$. Denote by $\mathcal{G}$ the associated graded algebra and denote by $\breve{\theta}_{i,r},\breve{b}_{\alpha,k}$ the images of $\theta_{i,r},b_{\alpha,k}$ in $\mathcal{G}$. By \eqref{def:cgr}, we have 
\begin{align*}
\breve{\theta}_{i,r,1}=\breve{\theta}_{i,r+1},\qquad
\breve{\beta}_{\alpha,m,1}=\breve{b}_{\alpha,m+1},\qquad
\breve{\beta}_{\alpha,0,l}=\breve{b}_{\alpha,l},
\qquad r,m \geq 0, l\leq 0, i\in \I^0,\alpha\in \cR^+.
\end{align*}
Then, by the above PBW basis for $U(\g[t,t^{-1}]^\omega)$,
the ordered monomials of $\{\breve{\theta}_{i,r,1},\breve{\beta}_{\alpha,m,1},\breve{\beta}_{\alpha,0,l}|$
$i\in \I^0, r,m \geq 0,l\leq 0,\alpha\in \cR^+ \}$ form a basis for $\mathcal{G}$. Hence, the ordered monomials of 
\[
\{ \theta_{i,r,1},\beta_{\alpha,m,1},\beta_{\alpha,0,l}|i\in \I^0, r,m \geq 0,l\leq0,\alpha\in \cR^+ \}
\]
form a basis for $U(\g[t,t^{-1}]^\omega) $. By Proposition~\ref{prop:climit}, such a monomial lies in $\kappa^\imath_{m_0}$ if and only if the sum of the second indices is at least $m_0$; then this subset of ordered monomials forms a spanning set for $\kappa^\imath_{m_0}$. Thus, $\ov{\theta}_{i,r,1},\ov{\beta}_{i,m,1},\ov{\beta}_{i,0,l}$ generate $\Gr_{\kappa^\imath}\g[t,t^{-1}]^\omega$. Since $\ov{\beta}_{i,0,l}=\ov{\beta}_{i,0,1}$ and $\ov{\theta}_{i,r,1}=0$ for $r$ even, the surjectivity for $\Phi_\imath$ follows.  
\end{proof}

\subsection{A reduced presentation for $U\big(\g[u]^{\check\omega}\big)$}

\begin{lemma}\label{lem:classical}
For $s,r$ even and $i\in \I^0$, we have $[x_{i,s},x_{i,r}]=0$.
\end{lemma}

\begin{proof}
%
We can assume that $s\ge r$ and both are even. The statement for $s=r$ is trivial. By \eqref{eq:classical3}, we have $[x_{i,s},x_{i,r}]=[x_{i,s-1},x_{i,r+1}]$, and the statement follows by a simple induction on $s-r$.
\end{proof}

We now give a reduced presentation for the algebra $U(\g[u]^{\check\omega})$ which has relations simpler than those in Proposition \ref{prop:crel}, i.e., we use the finite type Serre relations to replace the current Serre relations.

\begin{proposition}
\label{prop:reduced relation}
The algebra $U(\g[u]^{\check\omega})$ is generated by $t_{i,r},x_{i,m}$ for $i\in \I^0,r,m\geq 0$ with $r$ odd, subject to the relations \eqref{eq:classical1}--\eqref{eq:classical3} and the finite type Serre relations \eqref{eq:classical7a}--\eqref{eq:classical9} below:
\begin{align}
\label{eq:classical7a}
 [x_{i,0},x_{j,0}] 
&=0 \qquad\qquad\qquad\;\; \hskip 3cm (c_{ij}=0),
\\
\label{eq:classical7}
 \big[x_{i,0},[x_{i,0},x_{j,0}]\big] 
&= -x_{j,0} \qquad\qquad\; \hskip 2.9cm (c_{ij}=-1),
\\
\label{eq:classical8}
 \Big[x_{i,0},\big[x_{i,0},[x_{i,0},x_{j,0}]\big]\Big] &= -4  [x_{i,0},x_{j,0}] \qquad \hskip 2.65cm (c_{ij}=-2), 
\\
\label{eq:classical9}
 \bigg[x_{i,0},\Big[x_{i,0},\big[x_{i,0},[x_{i,0},x_{j,0}]\big]\Big]\bigg]
&= -10   \big[x_{i,0},[x_{i, 0},x_{j,0}]\big]
-9 x_{j,0}\qquad (c_{ij}=-3).
\end{align}
\end{proposition}

\begin{proof}
We have shown in Proposition \ref{prop:crel} that the algebra $U(\g[u]^{\check\omega})$ is generated by the same set of generators with relations \eqref{eq:classical1}--\eqref{eq:classical3} and the Serre relations \eqref{eq:classical4'}--\eqref{eq:classical6}.
Therefore, it suffices to establish the Serre relations \eqref{eq:classical4'}--\eqref{eq:classical6} for a pair $i\neq j$ from the relations \eqref{eq:classical1}--\eqref{eq:classical3} together with the finite type Serre relations \eqref{eq:classical7a}--\eqref{eq:classical9} for the same pair $i\neq j$.

We shall derive the Serre relation \eqref{eq:classical5} for $c_{ij}=-2$ from \eqref{eq:classical8} using the relations \eqref{eq:classical1}--\eqref{eq:classical3}. The proofs for the Serre relations in other cases are similar and will be skipped.

Fix $i,j\in \I^0$ such that $c_{ij}=-2$. We denote the difference of both sides of \eqref{eq:classical5} by
\begin{align*}
\mathbb{S}(k_1,k_2,k_3|r)
=\Sym_{k_1,k_2,k_3}\Big[x_{i,k_1},\big[x_{i,k_2},[x_{i,k_3},x_{j,r}]\big]\Big]-4\Sym_{k_1,k_2,k_3}  (-1)^{k_1+1} [x_{i,k_1+k_2+k_3},x_{j,r}].
\end{align*}
We shall show that  $\mathbb{S}(k_1,k_2,k_3|r)=0$, for any $k_1,k_2,k_3,r\geq 0$, by induction on the largest index of nonzero $k_i$'s. 

Denote by 
\[
\widetilde{t}_{j,r} = \frac{2t_{j,r}+t_{i,r}}{2}, \qquad \widetilde{t}_{i,r} = t_{i,r}+t_{j,r}, \qquad\text{ for $r$ odd}. 
\]
By \eqref{eq:classical2}, we have
\begin{align}
\label{eq:classical2'}
\begin{split}
[\widetilde{t}_{j,r},x_{j,k}]=x_{j,k+r},\quad
[\widetilde{t}_{j,r},x_{i,k}]=0, \quad 
[\widetilde{t}_{i,r},x_{i,k}]=x_{i,k+r},\quad
[\widetilde{t}_{i,r},x_{j,k}]=0.
\end{split}
\end{align}

Set $\ad x(y)=[x,y]$. It follows by \eqref{eq:classical8} that  $\mathbb{S}(0,0,0|0)=0$. By \eqref{eq:classical2'}, applying $\ad \widetilde{t}_{j,1}$ to $\bS(0,0,0,0)$ for a multiple of times, we have $\bS(0,0,0|r)=0$, for all $r\geq 0$, proving the base case. 

We next show inductively on $t$, for $1\le t\le 3$, that if $\bS(k_1,\ldots, k_{t-1},0,\ldots,0|r)=0$ for all $k_1,\ldots,k_{t-1}\geq 0$, then we have $\bS(k_1,\ldots, k_t,0,\ldots,0|r )=0$, for any $k_1,\ldots,k_{t}\geq 0$. We divide it into two cases.

\underline{Case (i)}. Suppose that $k_1,\ldots,k_t$ are all even. By Lemma~\ref{lem:classical}, $x_{i,k_1},\ldots,x_{i,k_t}$ commute with each other. By \eqref{eq:classical3}, $[x_{i,k},x_{j,0}]=[x_{i,0},x_{j,k}]$, for $k$ even. Hence, we have
\begin{align*}
\bS(k_1,\ldots,k_t,0,\ldots,0|r )
=\bS(k_1,\ldots,k_{t-1},0,\ldots,0 |k_t+r )=0.
\end{align*}

\underline{Case (ii)}.  Suppose that $k_i$ is odd for some $i$. Since $\mathbb{S}(k_1,k_2,k_3|r) $ is symmetric with respect to $k_1,k_2,k_3$, we can assume that $k_t$ is odd. Applying $\ad \widetilde{t}_{i,k_t}$ to $\bS(k_1,\ldots, k_{t-1},0,\ldots,0|r)$, we obtain that 
\begin{align*}
&\quad(4-t)\bS(k_1,\ldots, k_{t},0,\ldots,0|r )
+\bS(k_1+k_t,k_2\ldots, k_{t-1},0,\ldots,0|r )+\\ &+\bS(k_1,k_2+k_{t}\ldots, k_{t-1},0,\ldots,0|r )+\cdots+\bS(k_1,k_2\ldots, k_{t-1}+k_{t},0,\ldots,0|r )=0.
\end{align*}
The induction hypothesis implies that, in the left-hand side of the above equality, all terms except the first one are $0$; this forces the first term to be $0$.

Hence we conclude that $\bS(k_1,\ldots, k_t,0,\ldots,0|r )=0$, for any $k_1,\ldots,k_{t}\geq 0$, especially for $t=3.$ This proves \eqref{eq:classical5}. 
\end{proof}

\section{Affine $\imath$quantum groups}
 \label{sec:iQG}

In this section, we review and set up notations for the Drinfeld type presentation of an affine $\imath$quantum group of split type from \cite{LW21, Z22}. 

\subsection{Drinfeld type presentation}

Let $v$ be an indeterminate. Recall $d_i$ for $i\in \I$ from \S\ref{subsec:loop}. Define, for $m\in \Z, i\in \I$,
\[
v_i=v^{d_i}, \qquad
[m]_{i}=\frac{v_i^m-v_i^{-m}}{v_i-v_i^{-1}}, 
\qquad \qbinom{m}{s}_i=\frac{[m]_i!}{[s]_i![m-s]_i!}.
\]
Write $[A,B]=AB-BA$ and $[A,B]_a=AB-aBA$.  

Let $\U:=U_v(\hg)$ be the Drinfeld-Jimbo quantum group, i.e., $\U$ is the $\C(v)$-algebra generated by $E_i,F_i,K_i^{\pm 1},i\in \mathbb{I}$, subject to the relations
\begin{align*}
[E_i,F_j]&=\delta_{ij} \frac{K_i-K^{-1}_i}{v_i-v_i^{-1}},\qquad [K_i,K_j]=0,\\
K_i E_j&=v_i^{c_{ij}}E_j K_i,\qquad K_i F_j=v_i^{-c_{ij}}F_j K_i,
\end{align*}
and the standard quantum Serre relations which we skip here. The $\imath$quantum group of split type $\tUi$ is by definition the subalgebra of $\U$ generated by
\begin{equation} \label{Bi}
B_i=F_i-v_i^{-2}E_iK_i^{-1}, \qquad i\in\I.
\end{equation}
(We fix the parameters for the $\imath$quantum group $\Ui$ to be $-v_i^{-2}$ in \eqref{Bi} to facilitate the degeneration to twisted Yangians below.)

We recall from \cite{LW21,Z22} the Drinfeld type presentation for $\tUi$.

\begin{proposition} [\text{cf. \cite[Theorem 5.4]{LW21}, \cite[Theorem 3.3]{Z22}}]
 \label{prop:Dr}
The algebra $\tUi$ is isomorphic to the $\C(v)$-algebra generated by $B_{i,k},\TH_{i,m}$, for $i\in \I^0,k\in \Z,m>0$, subject to the following relations \eqref{idef1}--\eqref{idef3} for $i,j\in\I^0,k,l\in \Z$: 
\begin{align}
\label{idef1}
 [H_{i,k},H_{j,l}] &=0,
\\
\label{idef2}
[H_{i,k},B_{j,l}] &=\frac{[kc_{ij}]_{i}}{k}B_{j,l+k}-\frac{[kc_{ij}]_{i}}{k}B_{j,l-k},
\\
\label{idef3}
[B_{i,k},B_{j,l+1}]_{v_i^{-c_{ij}}}-v_i^{-c_{ij}}[B_{i,k+1},B_{j,l}]_{v_i^{c_{ij}}}
& =\delta_{ij}v_i^{-2} \big( \Theta_{i,l+1-k}-\Theta_{i,l-1-k}+ \Theta_{i,k+1-l}-\Theta_{i,k-1-l}\big),
\end{align}
and the Serre relations \eqref{SerreBB0}--\eqref{idef7}:  
\begin{equation} \label{SerreBB0}
[B_{i,k},B_{j,l}]=0, 
\qquad
\text{ for } c_{ij}=0;
\end{equation}
\begin{align}   
\label{idef4}
\begin{split}
&\mathrm{Sym}_{k_1,k_2} \big(B_{i,k_1} B_{i,k_2} B_{j,l} -[2]_i B_{i,k_1} B_{j,l} B_{i,k_2} + B_{j,l} B_{i,k_1} B_{i,k_2}\big)
\\
&=\mathrm{Sym}_{k_1,k_2}\bigg( -\sum_{p=0}^{\bigl\lfloor\frac{k_2-k_1-1}{2}\bigr\rfloor } (v_i^{2p+1}+v_i^{-2p-1}) [\Theta_{i,k_2-k_1-2p-1},B_{j,l-1}]_{v_i^{-2}} 
\\
&\qquad\quad- \sum_{p=1}^{\bigl\lfloor\frac{k_2-k_1 }{2}\bigr\rfloor} (v_i^{2p}+v_i^{-2p})[B_{j,l},\Theta_{i,k_2-k_1-2p}]_{v_i^{-2}}-[B_{j,l},\Theta_{i,k_2-k_1}]_{v_i^{-2}} \bigg),
\text{ for } c_{ij}=-1;
\end{split}
\end{align}
\begin{align}\label{idef6}
\sum_{s=0}^3 (-1)^s \qbinom{3}{s}_{i} B_{i,k}^{3-s} B_{j,l} B_{i,k}^s=-v_i^{-1} [2]_{i}^2 [B_{i,k}, B_{j,l}],
\qquad
\text{ for } c_{ij}=-2;
\end{align}
and 
\begin{align}
\label{idef7}
\sum_{s=0}^4(-1)^s
\qbinom{4}{s}_{i}
B_{i,k}^{4-s}B_{j,l} B_{i,k}^s& = -v_i^{-1}(1+[3]_{i}^2)( B_{j,l} B_{i,k}^2+ B_{i,k}^2 B_{j,l})\\\notag
&\quad  +v_i^{-1}[4]_{i} (1+[2]_{i}^2) B_{i,k} B_{j,l} B_{i,k} -v_i^{-2}[3]^2_{i} B_{j,l},
\qquad
\text{ for } c_{ij}=-3.
\end{align}
Here
\begin{align}
 \label{ThetaH}
1+(v_i-v_i^{-1})\sum_{k\geq 1}\Theta_{i,k}u^k&=\exp \big((v_i-v_i^{-1})\sum_{k\geq 1}H_{i,k}u^k \big),
\end{align}
and $\Theta_{i,0}=(v_i-v_i^{-1})^{-1},\Theta_{i,m}=0$ if $m<0.$
\end{proposition}

\begin{remark}
The $\Theta_{i,l}$ here corresponds to the imaginary root vectors introduced by Baseilhac-Kolb \cite{BK20}, which were denoted by $\acute{\Theta}_{i,l}$ in \cite{LW21}. Alternative imaginary root vectors were used in \cite[\S 2.5]{LW21} (also see \cite[\S 3.4]{LW21}) which led to a presentation of $\Ui$ of split ADE type with relations different from \eqref{idef3}--\eqref{idef4}.
It was realized in \cite{Z22} that the Serre relations \eqref{idef6}--\eqref{idef7} suffice for a presentation of $\Ui$ of split BCFG type.

\end{remark}

\subsection{The classical limit}

When identifying the classical limit at $v\mapsto 1$ of $\mathbf{U}^\imath$ with $U(\hg^\omega)$, we identify the generators of $\tUi$ with those of $U(\hg^\omega)$ via 
\begin{equation}\label{classlimit}
B_{i,k} \rightsquigarrow b_{i,k}, \qquad \TH_{i,m} \rightsquigarrow \theta_{i,m},\qquad \text{for } k\in \mathbb{Z}, m>0, i\in\I^0,
\end{equation}
where $b_{i,k},\theta_{i,m}$ are defined in \eqref{def:loop}; see \cite[\S 3.6]{LW21}.

For each $\alpha\in \cR^+$, we fix $i_1,\ldots, i_m\in \I^0$ such that $f_{\alpha}:=[f_{i_1},[f_{i_2},\cdots[f_{i_{m-1}},f_{i_m}]\cdots]]$ is a nonzero root vector in $\g_{-\alpha}$. Define
\begin{align}
  \label{def:Balpha}
B_{\alpha,k}=\Big[B_{i_1,0},\big[B_{i_2,0},\cdots[B_{i_{m-1},0},B_{i_m,k}]\cdots\big]\Big].
\end{align}

Let $\A$ be the localization of $\C[v]$ with respect to the ideal $(v-1)$ and $\tUiA$ be the $\A$-subalgebra of $\tUi$ generated by the elements $B_{i,k},\TH_{i,m}$ for $k\in \Z,m> 0$. Let $\U_\A$ be the $\A$-subalgebra of $\U$ generated by $E_i,F_i,K_i^{\pm1}, \frac{K_i-1}{v-1}$. It is known (cf. \cite[Corollary 10.3]{Ko14}) that $\tUiA\subset \U_\A$. The following proposition is an analog of \cite[Proposition 2.1]{GM12}. 

\begin{proposition}[cf. \cite{Ko14}]
\label{prop:GM}
We have an algebra isomorphism $\tUiA / (v-1) \tUiA \stackrel{\cong}{\longrightarrow} U(\hg^\omega)$. Moreover, $\tUiA$ is a free $\A$-module isomorphic to $U(\hg^\omega) \otimes_{\C} \A$. 
\end{proposition}

\begin{proof}
The first statement was established in \cite[Theorem 10.8]{Ko14}. We prove the second statement. It is well known that $\U_\A$ is a free $\A$-module with a monomial $\A$-basis. Since $\A$ is a principal ideal domain, the $\A$-submodule $\Ui_\A$ of $\U_\A$ is free. A monomial basis for $\tUi$ was constructed in \cite[Proposition 6.2]{Ko14}, and one can use similar arguments therein to establish a monomial basis for $\tUi_\A$. The classical limit of the monomial basis forms a basis of $U(\hg^\omega)$, and this implies that $\tUiA$ is isomorphic to $U(\hg^\omega) \otimes_{\C} \A$ as $\A$-modules.
\end{proof}

\subsection{Defining relations via generating functions}

Define generating functions in a variable $z$:
\begin{align}\label{eq:Genfun}
\begin{cases}
\B_{i}(z) =\sum_{k\in\Z} B_{i,k}z^{k},
\\
 \bTH_{i}(z)  =(v_i-v_i^{-1})^{-1}+ \sum_{m > 0}\Theta_{i,m}z^{m},
\\
\bH_i(z)=\sum_{m\geq 1} H_{i,m} z^m,
\\
\bDel(z)=\sum_{k\in\Z}  z^k.
\end{cases}
\end{align}
The identity \eqref{ThetaH} can be reformulated as 
\[
(v_i-v_i^{-1})\bTH_{i}(z)=\exp\big( (v_i-v_i^{-1})\bH_i(z)\big).
\]

The Drinfeld type relations \eqref{idef1}--\eqref{idef4} for $\Ui$ can be reformulated in terms of generating functions as \eqref{iDRG1}--\eqref{iDRG4} (see \cite{LW21}):
\begin{align}
& \bTH_i(z) \bTH_j(w) =\bTH_j(w) \bTH_i(z),
\label{iDRG1}
\\
& \B_j(w)  \bTH_i(z)
 = \bigg(
  \frac{1 -v_i^{c_{ij}}zw^{-1}}{1 -v_i^{-c_{ij}}zw^{-1}} \cdot \frac{1 -v_i^{-c_{ij}}zw}{1 -v_i^{c_{ij}} zw }
  \bigg)
\bTH_i(z) \B_j(w), \label{iDRG2}
\\
&(v_i^{c_{ij}}z -w) \B_i(z) \B_j(w) +(v_i^{c_{ij}}w-z) \B_j(w) \B_i(z)=0, \qquad \text{ if }i\neq j,\label{iDRG3a}
\\\notag
&(v_i^2z-w) \B_i(z) \B_i(w) +(v_i^{2}w-z) \B_i(w) \B_i(z)
\\\label{iDRG3b}
&\qquad\qquad\qquad\qquad\quad =v_i^{-2} \bDel(zw)(z-w) \big( \bTH_i(w) -\bTH_i(z) \big),
\\\notag
&\Sym_{w_1,w_2}\big\{\B_{i}(w_1)\B_{i}(w_2)\B_{j}(z)
-[2]_{v_i}\B_{i}(w_1)\B_{j}(z)\B_{i}(w_2)+\B_{j}(z)\B_{i}(w_1)\B_{i}(w_2)\big\}
\\\label{iDRG4}
&\qquad= \Sym_{w_1,w_2}\bDel(w_1w_2)
\frac{[2]_{i} z(w_2-w_1)}{(w_1-v_i^2 w_2)(w_1-v_i^{-2}w_2)}[\bTH_i(w_2),\B_{j}(z)]_{v_i^{-2}} 
\\\notag
&\qquad\quad +\Sym_{w_1,w_2} \bDel(w_1w_2)\frac{w_2^2-w_1^2}{(w_1-v_i^2 w_2)(w_1-v_i^{-2}w_2)}[\B_{j}(z),\bTH_i(w_2)]_{v_i^{-2}},
\text{ for } c_{ij}=-1.
\end{align}

The following Serre type relation for $c_{ij}=-2$ is given in \cite[(5.21)]{Z22}
\begin{align}
\label{gen4}&\big((w_1+w_2+w_3)(w_1^{-1}+w_2^{-1}+w_3^{-1})-[3]^2_{v_i}\big)\mathbb{S}(w_1,w_2,w_3|z) \\\notag
&= [3]_{v_i}z^{-1}\Sym_{w_1,w_2,w_3}  \bDel(w_2w_3) ( w_2-w_3)\\\notag
&\quad\times\bigg(v_i[2]_{v_i}\big[\B_i(w_1),[\B_j(z),\bTH_i(w_2)]_{v_i^{-4}}\big]-\big[[ \B_j(z),\B_i(w_1)]_{v_i^{-2}},\bTH_i(w_2)\big]\bigg)\\\notag
&\quad+\Sym_{w_1,w_2,w_3} \bDel(w_2w_3)( w_3^{-1}-w_2^{-1})(w_1+w_2+w_3)\\\notag
&\quad\times\bigg(-v_i[2]_{v_i}\big[[\bTH_i(w_2),\B_j(z)]_{v_i^{-4}}, \B_i(w_1)\big]+\big[\bTH_i(w_2),[\B_i(w_1), \B_j(z)]_{v_i^{-2}}\big]\bigg).
\end{align}

For later use, we reformulate some Drinfeld type relations for $\Ui$ (see Proposition \ref{prop:Dr}) in different forms.

\begin{proposition}
\label{prop:equiv}
\begin{enumerate}
    \item 
The relation \eqref{idef3} for $m:=l-k-1\ge 1$ can be reformulated as 
\begin{align}
\label{idef3'}
[B_{i,k},B_{j,m+k+2}]_{v_i^{-c_{ij}}}-v_i^{-c_{ij}}[B_{i,k+1},B_{j,m+k+1}]_{v_i^{c_{ij}}} 
=\delta_{ij}v_i^{-2} \left( \Theta_{i,m+2} -\Theta_{i,m} \right).
\end{align}
\item 
The relation \eqref{idef2} is equivalent to the following relation:
\begin{align}
\label{idef2'}
\begin{split}
[B_{j,l},\Theta_{i,k+2}] &+[B_{j,l},\Theta_{i,k} ]-v_i^{c_{ij}}B_{j,l-1}\Theta_{i,k+1} \\
&-v_i^{-c_{ij}}B_{j,l+1}\Theta_{i,k+1}+v_i^{c_{ij}}\Theta_{i,k+1}B_{j,l+1}+v_i^{-c_{ij}}\Theta_{i,k+1}B_{j,l-1} =0.
\end{split}
\end{align}
\end{enumerate}
\end{proposition}

\begin{proof}
    (1) is clear. 

    (2) follows by some manipulations of generating functions $\B_j(z), \bH_i(z)$ and $\bTH_i(z)$ in \eqref{eq:Genfun}. We refer to \cite[\S2.4]{LW21} or \cite[Lemma 4.8]{LWZ24} for variants and detailed proofs. 
\end{proof}

\section{Drinfeld type presentations for twisted Yangians}
\label{sec:main}

We formulate twisted Yangians $\Y$ of split type (excluding $G_2$) in Drinfeld presentations and show that they are isomorphic to the degenerations of affine $\imath$quantum groups of split type. A PBW basis for $\Y$ is established and this allows us to conclude that $\Y$ is a flat deformation of the corresponding twisted current algebra $U(\g [u]^{\check\omega})$. Finally, we provide a second presentation of $\Y$ which only uses finite type Serre relations.

\subsection{Twisted Yangians of split type}
\label{degBK}

 We shall denote
 \[
 \delta_{k,ev} = 
 \begin{cases}
 1 & \text{ if $k$ is even},
 \\
 0 & \text{ if $k$ is odd}. 
 \end{cases}
 \]
We also write $\{A,B\}=AB+BA$. Recall that $(c_{ij})_{i,j\in \I^0}$ denotes the Cartan matrix for $\g$. We exclude the type $G_2$ in this subsection, and will return to it in Section~\ref{sec:G2}.

\begin{definition}
\label{def:YN}
The {\em twisted Yangian of split type $(\I, \text{Id})$}, denoted $\Y$, is defined to be the $\C[\h]$-algebra generated by $\H_{i,s},\X_{j,r}$, for $s\in 2\N+1,r\in\N,i,j\in \mathbb{I}^0$, subject to the following relations \eqref{ty0}--\eqref{ty2}:
\begin{align}\label{ty0}
[\H_{i,s},\H_{j,r}] &=0,
\\
\label{ty5}
[\H_{i,1},\X_{j,r}]& =2  c_{ij}\X_{j,r+1},
\\
\label{ty1}
[\X_{j,r},\H_{i,s+2}]-[\X_{j,r+2},\H_{i,s}]
&=-\h d_i c_{ij}\{ \X_{j,r+1},\H_{i,s} \}+ \frac{\h^2 d_i^2 c^2_{ij}}{4}[\X_{j,r},\H_{i,s}] ,
\\
\label{ty2}
[\X_{i,s},\X_{j,r+1}]-[\X_{i,s+1},\X_{j,r}] &+\frac{\h d_i c_{ij}}{2}\{\X_{j,r},\X_{i,s}\}
=2\delta_{ij} (-1)^s \H_{i,r+s+1},
\end{align}
and the Serre type relations \eqref{ty3}--\eqref{ty6}, for $0\ge c_{ij}\ge -2$:
\begin{equation}\label{ty3}
[\X_{i,s},\X_{j,r}]= 0 \qquad\qquad\qquad\qquad (c_{ij}=0),
\end{equation}
\begin{equation}\label{ty4}
\begin{aligned}
&\Sym_{k_1,k_2}\Big[b_{i,k_1} ,\big[b_{i,k_2}, b_{j,r} \big] \Big]
\\
&=\delta_{k_1+k_2,ev}( -1)^{k_1+1} \left(2 \left( \frac{\h d_i}{2}\right)^{k_1+k_2} \X_{j,r} +\sum_{t=0}^{ \frac{k_1+k_2}{2}-1 } \left( \frac{\h d_i}{2}\right)^{2t} [\X_{j,r+1}, \H_{i,k_1+k_2-2t-1}] \right.
\\
&\left. \qquad\qquad \qquad\qquad  + \sum_{t=0}^{ \frac{k_1+k_2}{2}-1 } 2^{-2t} (\h d_i)^{2t+1}  \{ \H_{i,k_1+k_2-2t-1}, \X_{j,r} \} \right)
\quad (c_{ij}=-1),
\end{aligned}
\end{equation}
and
\begin{align}
& \Sym_{k_1,k_2,k_3}\Big[b_{i,k_1} ,\big[b_{i,k_2},[b_{i,k_3}, b_{j,r} ] \big]\Big]
\notag \\ 
 &= \mathrm{Cyc}_{k_1,k_2,k_3}  \sum_{s\geq 0}\sum_{p=0}^s(-1)^{k_3+p+1}\binom{s}{p} 3^{-p-1}\h^{2s-2p}\times
\notag  \\ 
&\quad\times\Big\{ 3\big[[b_{i,k_1+2p},b_{j,r+1}],h_{i,k_2+k_3-2s-1}\big]
+\big[[b_{i,k_1+2p+1},b_{j,r}],h_{i,k_2+k_3-2s-1}\big] +
\notag \\ 
&\qquad +6\big[b_{i,k_1+2p},[b_{j,r+1},h_{i,k_2+k_3-2s-1}]\big]
+2\big[b_{i,k_1+2p+1},[b_{j,r},h_{i,k_2+k_3-2s-1}]\big]+
\label{ty6}  \\ 
&\qquad +3 \h d_i \big[\big\{b_{i,k_1+2p},h_{i,k_2+k_3-2s-1}\big\},b_{j,r} \big]
+9 \h d_i \big[b_{i,k_1+2p},\{h_{i,k_2+k_3-2s-1},b_{j,r} \}\big]\Big\}
\notag \\ 
&\qquad + 8  \mathrm{Cyc}_{k_1,k_2,k_3} \delta_{k_2+k_3,ev}
\sum_{p= 0}^{\frac{k_2+k_3}{2}}(-1)^{k_3+p+1} \binom{\frac{k_2+k_3}{2}}{p}3^{-p}\h^{k_2+k_3-2p}\big[b_{i,k_1+2p},b_{j,r}\big] \quad (c_{ij}=-2),
\notag
\end{align}
where $\H_{i,s}=0$ if $s$ is even or $s<0$. 
\end{definition}

A variant of $\Y$ would be a $\C$-algebra by setting $\hbar=1$ in the relations \eqref{ty0}--\eqref{ty6}.

\begin{remark}  \label{rem:noG2}
    Definition~\ref{def:YN} excludes the twisted Yangian of type $G_2$, as it does not include  the Serre relations for $c_{ij}=-3$. See however Section~\ref{sec:G2} for an alternative approach for twisted Yangian of type $G_2$ similar to the reduced presentation of $\Y$ in Theorem~\ref{thm:Yreduced} below. 
\end{remark}

\begin{remark}
Replacing the Cartan matrix $C^0$ (and respectively, the indexing set $\I^0$ for generators) in the Drinfeld presentations of twisted Yangians above by the affine Cartan matrix $C$ (and respectively, $\I$), we obtain a definition of {\em affine twisted Yangians} of split types. 
\end{remark}

\begin{lemma} \label{lem:equiv}
\begin{enumerate}
    \item 
For $i=j$, the relation \eqref{ty2} is equivalent to the following relation:
\begin{align}\label{bbalt}
[b_{i,s},b_{i,r+1}] =-\h d_i\sum_{t=0}^{r-s} \X_{i,r-t}\X_{i,s+t}+(-1)^s h_{i,r+s+1},
\qquad \text{ for } s\leq r.
\end{align}
\item
The relation \eqref{ty1} is equivalent to the following relation:
\begin{align}\label{bhalt}
\begin{split}
[\X_{j,r},\H_{i,s}]=&-\h d_i c_{ij}\sum_{t=1}^{\lfloor s/2\rfloor}( \X_{j,r+2t-1}\H_{i,s-2t} + \H_{i,s-2t}\X_{j,r+2t-1})
\\
&+  \frac{\h^2 d_i^2 c^2_{ij}}{4}
\sum_{t=1}^{\lfloor s/2\rfloor}[\X_{j,r+2t-2},\H_{i,s-2t}] -2 c_{ij}\X_{j,r+s}.
\end{split}
\end{align}
\end{enumerate}
\end{lemma}

\begin{proof}
(1) The identity \eqref{ty2} for $i=j$ follows from subtracting the formula \eqref{bbalt} by its variant with different indices for $[b_{i,s+1},b_{i,r}].$ On the other hand, the identity \eqref{bbalt} follows from \eqref{ty2} by a simple induction on $r-s$. The proof for (2) is entirely similar. 
\end{proof}

\subsection{$\Y$ as a deformation of $U(\g[u]^{\check\omega})$ over $\C[\hbar]$}

It follows from Definition~\ref{def:YN} that $\Y$ is an $\N$-graded $\C[\hbar]$-algebra by assigning degrees $\deg\h=1,\deg b_{i,r}=r, \deg h_{i,s} =s$, for all admissible $i,r,s$. Following a referee's suggestion, we are ready to formulate the following deformation statement now, and Proposition \ref{prop:PBWbasis} below is an immediate consequence of this. 

\begin{proposition} \label{prop:deform}
    The twisted Yangian $\Y$ is an $\N$-graded $\C[\hbar]$-algebra deformation of $U(\g[u]^{\check\omega})$. 
\end{proposition}

\begin{proof}
    By Proposition~\ref{prop:crel}, Definition~\ref{def:YN} and Lemma~\ref{lem:equiv}, we have a homomorphism
\begin{align*}
\jmath\colon U(\g[u]^{\check\omega}) & \longrightarrow \Y /\hbar \Y,
\qquad
2t_{i,s} \mapsto h_{i,s},  \quad
x_{i,r} \mapsto \ov{b}_{i,r},
\end{align*}
for all admissible $i,r,s$. Now viewing $U(\g[u]^{\check\omega})$ as a $\C[\hbar]$-algebra, via the evaluation $\C[\hbar] \rightarrow \C$, $\hbar\mapsto 0$, we obtain a $\C[\hbar]$-algebra homomorphism $\Y \rightarrow U(\g[u]^{\check\omega})$, $h_{i,s} \mapsto 2t_{i,s}, b_{i,r} \mapsto x_{i,r}$. This homomorphism factors through a homomorphism $\Y/\hbar\Y \rightarrow U(\g[u]^{\check\omega})$, which is clearly the inverse to the homomorphism $\jmath$. Hence $\jmath$ is an isomorphism. 
\end{proof}

 For each $\alpha\in \cR^+$, we fix $i_1,\ldots, i_m\in \I^0$ such that $f_{\alpha}:=[f_{i_1},[f_{i_2}, \cdots [f_{i_{m-1}},f_{i_m}]\cdots]]$ is a nonzero root vector in $\g_{-\alpha}$. Define
\begin{align}
\label{def:bfalpha}
\X_{\alpha,r}=\Big[\X_{i_1,0},\big[\X_{i_2,0},\cdots[\X_{i_{m-1},0},\X_{i_m,r}]\cdots\big]\Big].
\end{align}

\begin{proposition}
  \label{prop:PBWbasis}
The ordered monomials of $\{\X_{\alpha,r},\H_{i,s} \mid \alpha\in \cR^+,i\in \I^0,r\in \N, s\in 2\N+1\}$ (with respect to any fixed total ordering) form a spanning set of $\Y$.
\end{proposition}

\begin{proof}
Follow by Proposition \ref{prop:deform} and the well-known spanning set (actually a basis) for $U(\g[u]^{\check\omega})$.
\end{proof}

\begin{remark}
In the type A case, root vectors $\ttb_{ij}^{(r)}$ of $\Y_{N}$, for $i>j$ and $r\geq 0$, are introduced in \cite[(5.9)]{KLWZ23}. For $\alpha=\alpha_{j}+\alpha_{j+1}+\ldots +\alpha_{i-1}$, the element $\X_{\alpha,r}$ is identified with $\ttb_{ij}^{(r)}$, up to a sign. Then Proposition~\ref{prop:PBWbasis} for type A also follows by \cite[\S5.2, Proof of Theorem~ 5.1]{KLWZ23}.
\end{remark}

\subsection{Presentation via generating functions}
Define
\begin{align}
b_i(u)=\h d_i\sum_{r\geq 0} \X_{i,r} u^{-r-1}, \qquad h_i(u)=1+\h d_i\sum_{r\geq 0} \H_{i,r} u^{-r-1}.
\end{align}
Defining relations \eqref{ty0}--\eqref{ty6} for $\Y$ can be reformulated in generating function form as \eqref{gconj1}--\eqref{gconj7} below:
\begin{align}
\label{gconj1}
[h_i(u),h_j(v)]&=0,
\\\notag
 (u^2-v^2)[h_i(u),b_j(v)]&=d_i c_{ij}\h v \{h_i(u),b_j(v)\}
 +\frac{\h^2}{4}d_i^2 c_{ij}^2[h_i(u),b_j(v)]
 \\&\quad -\h d_j[h_i(u),b_{j,1}]-v\h d_j[h_i(u),b_{j,0}]-d_i d_j c_{ij}\h^2\{h_i(u),b_{j,0}\}, 
 \label{gconj2}
\\\notag
(u-v)[b_i(u),b_j(v)]&=\frac{d_i c_{ij}\h}{2}\{b_i(u),b_j(v)\}
+\h\big(d_i[b_{i,0},b_j(v)]-d_j[b_i(u),b_{j,0}]\big)
\\&\quad -\delta_{ij}  \h d_i\frac{u-v}{u+v}\big(h_i(u)-h_i(v)\big)-\delta_{ij}d_i\h\big( h_i(u)+ h_j(v)-2\big),
\label{gconj4}
\\\label{gconj5}
[b_i(u),b_j(v)]&=0, \qquad\qquad (c_{ij}=0),
\end{align}
\begin{align}\notag
&\Sym_{u_1,u_2}\Big[b_{i}(u_1),\big[b_{i}(u_2), b_{j}(t) \big]\Big]
\\
&=\frac{4\h d_i}{u_1 + u_2}\Sym_{u_1,u_2}  \frac{u_2(t-\h d_i)\, h_i(u_2) b_j(t)- u_2(t+\h d_i)\, b_j(t)h_i(u_2) }{4u_2^2-d_i^2\h^2}
\quad (c_{ij}=-1),
\label{gconj6} 
\end{align}
and  
\begin{align}\notag
&\Sym_{u_1,u_2,u_3}\Big[b_{i}(u_1),\big[b_{i}(u_2),[b_i(u_3), b_{j}(t)] \big]\Big]
\\\notag
&=\Sym_{u_1,u_2,u_3} \frac{-u_2 \h d_i}{(u_2+u_3)(3u_2^2-3\h^2+u_1^2)} 
\\\notag
&\quad \times \Big\{ (3t+u_1) \big[[b_i(u_1),b_j(t)],h_i(u_2)\big]+2(3t+u_1) \big[b_i(u_1),[b_j(t),h_i(t)\big] 
\\
&\qquad +3 \h d_i \big[\{b_i(u_1),h_i(u_2)\},b_j(t) \big]+9 \h d_i \big[b_i(u_1),\{h_i(u_2),b_j(t) \}\big]
\Big\} \quad (c_{ij}=-2).
\label{gconj7}
\end{align}

\subsection{From affine $\imath$quantum groups to twisted Yangians}

We shall define a filtration on $\tUiA$ analogous to the one in \cite[Section 2]{CG15} on an affine quantum group. 

Let $\psi$ denote the following composite homomorphism (see Proposition~\ref{prop:GM})
\begin{align*}
\psi: \tUiA\twoheadrightarrow \tUiA/(v-1)\tUiA
\stackrel{\cong}{\longrightarrow} U(\widehat{\mathfrak{g}}^\omega).
\end{align*}
For $m\geq 0$, denote by $\mathrm{K}_m$ the Lie ideal of $\hg^\omega$ spanned by 
\[
g\otimes t^s(t-1)^m+\omega_0(g)\otimes t^{-s-m}(1-t)^m, \forall s\in \Z,g\in \g. 
\]
Denote 
\begin{align} \label{Wspan}
    W := \C\text{-span}  \{B_{\alpha,k},\TH_{i, \ell}
\mid \alpha\in \cR^+, i\in \I^0,k\in \Z,\ell > 0 \} \subset \tUiA. 
\end{align}

Set $\mathcal{K}_m$ to be the two-sided ideals of $\tUiA$ generated by $\psi^{-1}(\mathrm{K}_m)\cap W$. Define $\mathbf{K}_m$ to be the $\C$-sum of ideals $(v-v^{-1})^{m_0} \mathcal{K}_{m_1}\cdots \mathcal{K}_{m_r}$ such that $m_0+m_1+\ldots +m_r\geq m$.

Define $\Gr_{\mathbf K} \tUi$ to be the $\C$-algebra
\begin{align}
 \label{GrK}
\Gr_{\mathbf K} \tUi:=\bigoplus_{m\geq 0} \mathbf{K}_m/\mathbf{K}_{m+1}.
\end{align}
Then $\Gr_{\mathbf K} \tUi$ can be viewed as a $\C[\h]$-algebra by setting $\h:=\ov{v-v^{-1}}\in \bK_1/\bK_2$.
Define, for $l\geq 1, k\in\mathbb{Z}$ and $r\ge 0$,
\begin{equation}\label{def:HX}
\H_{i,r,l}=\sum_{s=0}^r (-1)^{r-s}\binom{r}{s}\Theta_{i,s+l},
\qquad 
\X_{i,r,k}=\sum_{s=0}^r (-1)^{r-s}\binom{r}{s}B_{i,s+k}.
\end{equation}

\begin{lemma}\label{lemma1}
We have $\X_{i,r+1,k}=\X_{i,r,k+1}-\X_{i,r,k}$ and $ \H_{i,r+1,k}=\H_{i,r,k+1}-\H_{i,r,k}$.
\end{lemma}

\begin{proof}
    Follows by the identity $\binom{r+1}{s}=\binom{r}{s}+\binom{r}{s-1}$ and the definitions in \eqref{def:HX}.
\end{proof}

\begin{lemma}\label{lemma2}
We have
\begin{itemize}
\item[(a)] $\X_{i,r,k}\in \bK_r$;
\item[(b)] $\H_{i,r,k}\in \bK_r$ for $r \ge 0$ odd, and $\H_{i,r,k}\in \bK_{r+1}$ for $r \ge 0$ even.
\end{itemize}
\end{lemma}

\begin{proof}
The $\psi$-images of $\X_{i,r,k},\H_{i,r,k}$ are given by $\beta_{i,r,k},\theta_{i,r,k}$, respectively. It is also clear from the definition \eqref{def:HX} that $\X_{i,r,k},\H_{i,r,k}\in W$.
By Proposition~\ref{prop:climit}, we have that (a) $\beta_{i,r,k}\in \mathrm{K}_r$; (b) $\theta_{i,r,k}\in \mathrm{K}_r$ for $r$ odd, and $\theta_{i,r,k}\in \mathrm{K}_{r+1}$ for $r$ even. Then the desired statements follow from the definition of $\bK_r$.
\end{proof}

Write $\ov{\X}_{i,r,k},\ov{\H}_{i,r,k}$ for the images of  $\X_{i,r,k},\H_{i,r,k}$ in $\bK_r/\bK_{r+1}$.

\begin{corollary}\label{lemma3}
We have
\begin{itemize}
\item[(a)] $\ov{\X}_{i,r,k+1}=\ov{\X}_{i,r,k}$ in $\Gr_{\mathbf{K}} \tUi$;
\item[(b)] $\ov{\H}_{i,r,k+1}=\ov{\H}_{i,r,k}$ in $\Gr_{\mathbf{K}} \tUi$ for $r\geq 0$ odd, and $\ov{\H}_{i,r,k}=0$ in $\Gr_{\mathbf{K}} \tUi$ for $r\geq 0$ even.
\end{itemize}
\end{corollary}

\subsection{The isomorphism between $\Y$ and $\Gr_{\mathbf{K}}\tUi$}

\begin{theorem}\label{thm:qiso}
There is an algebra isomorphism
\begin{align}
  \label{Phimap}
\begin{split}
\Phi: & \Y \longrightarrow \Gr_{\bK}\tUi, 
\\
& \X_{i,r}  \mapsto \ov{\X}_{i,r,1},\quad
\H_{i,s} \mapsto \ov{\H}_{i,s,1}, 
\end{split}
\end{align}
for $i\in \I^0,r\in \N, s\in 2\N+1$. 
\end{theorem}

\begin{proof}
In the subsequent Sections \ref{sec:verify}--\ref{sec:gen}, we will show that the defining relations \eqref{ty0}--\eqref{ty6} for $\Y$ are satisfied for $\ov{\X}_{i,r,1},\ov{\H}_{i,s,1}$ in $\Gr_{\bK}\tUi$. 
Hence $\Phi$ is an algebra homomorphism.

We now show that the map $\Phi$ is surjective. For each $\alpha\in \cR^+$, we fix $i_1,\ldots, i_m\in \I^0$ such that $f_{\alpha}:=[f_{i_1},[f_{i_2},\cdots[f_{i_{m-1}},f_{i_m}]\cdots]]$ is a nonzero element in $\g_{-\alpha}$. Define
\begin{align}
\X_{\alpha,r,k}=\Big[\X_{i_1,0,0},\big[\X_{i_2,0,0},\cdots[\X_{i_{m-1},0,0},\X_{i_m,r,k}]\cdots\big]\Big].
\end{align}
The $\psi$-image of $\X_{\alpha,r,k}$ is identified with $\beta_{\alpha,r,k}$; see Section~\ref{dege:notwisted} for the definition of $\beta_{\alpha,r,k}$. Using arguments similar to the proof of Proposition~\ref{prop:crel}, one can show that $\psi(\X_{\alpha,r,k}),\psi(\H_{i,r,k})$ \big(respectively, $\psi(\X_{\alpha,r,k})$\big) form a spanning set of $\mathrm{K}_r$ for $r$ odd (respectively, even). Moreover, these elements $\X_{\alpha,r,k},\H_{i,r,k}$ lie in $W$ by definition.

On the other hand, note that for any $X\in\psi^{-1}(\mathrm{K}_r)\cap W$, there exists an element $Y$ in the $\C$-span of $\X_{\alpha,r,k},\H_{i,r,k}$ such that $X-Y\in (v-1)\tUiA$, since the kernel of $\psi$ is $(v-1)\tUiA$. It follows that $X-Y\in (v-1)\tUiA \cap W=\{0\}$. Hence, we have shown that $\psi^{-1}(\mathrm{K}_m)\cap W$ is spanned over $\C$ by elements $\X_{\alpha,r,k},\H_{i,r,k}$. Therefore, $\bK_m$ is spanned by the elements of the form $f(v)(v-v^{-1})^{m_0}\mathcal{M}$, where $f(v)$ is not divisible by $v-1$ and
\begin{align}  \label{M}
\mathcal{M}=\X_{\alpha_1,r_1,k_1}\cdots \X_{\alpha_p,r_p,k_p} \H_{i_{1},s_{1},l_{1}}\cdots \H_{i_{q},s_{q},l_{q}}
\end{align}
is a monomial such that $r_1+\ldots +r_p+s_1+\ldots +s_q+m_0\geq m$. By Corollary~\ref{lemma3}, we deduce that $\bK_m/\bK_{m+1}$ is spanned by the following elements
\begin{align}
f(1)\h^{m_0} \ov{\X}_{\alpha_1,r_1,1}\cdots \ov{\X}_{\alpha_p,r_p,1} \ov{\H}_{i_{1},s_{1},1}\cdots \ov{\H}_{i_{q},s_{q},1},\qquad r_1+\ldots r_p+s_1+\ldots +s_q+m_0\geq m.
\end{align}
Hence the surjectivity of $\Phi$ follows.

We next show that the map $\Phi$ is injective. By Proposition~\ref{prop:PBWbasis}, the ordered monomials of $\X_{\alpha,r},\H_{i,s}$ for $\alpha\in \cR^+,i\in \I^0,r\in \N, s\in 2\N+1$ form a $\C[\h]$-spanning set of $\Y$. It remains to show that their images form a basis in $\Gr_{\mathbf{K}}\tUi$. In fact, we will adapt the arguments in \cite[Proof of Theorem 3.2]{GM12} to establish a $\C$-basis of $\bK_m/\bK_{m+1}$ for each $m$.

Let $\mathcal{M}$ be an ordered monomial of $\X_{\alpha,r,1},\H_{i,s,1}$ for $r,s\geq 0,s$ odd. The element $\psi(\mathcal{M})$ has the form $\beta_{\alpha_1,r_1,1}\cdots \beta_{\alpha_p,r_p,1} \theta_{i_{1},s_{1},1}\cdots \theta_{i_{q},s_{q},1}$; these monomials form a PBW basis of $U(\hg^\omega)$. By Proposition~\ref{prop:GM}, such ordered monomials of $\X_{\alpha,r,1},\H_{i,s,1}$ form an $\A$-basis of $\tUi_\A$. We set $\h$ to be $ v-v^{-1}$ in $\A=\C[v]_{(v-1)}$.  Hence, the set
$\{\h^m \mathcal{M} \}$ is linearly independent over $\C$ in $\tUi_\A$. 

We determine when an element $\h^m \mathcal{M}$ (with $\mathcal{M}$ in \eqref{M}) lies in $\bK_m$. If $\h^s \mathcal{M}\in \bK_m$, then $\h^{s-t} \mathcal{M}\in \bK_{m-t}$. We define the degree of $\mathcal{M}$ by the sum of second indices of $\X_{\alpha,r,1},\H_{i,s,1}$ appearing in $\mathcal{M}$. If $\mathcal{M}$ has degree at least $m$, then $\mathcal{M}\in \bK_m$. Conversely, if $\mathcal{M}\in \bK_m$, then the image of $\mathcal{M}$ under $\tUiA\twoheadrightarrow U(\hg^\omega)\hookrightarrow U(\hg)\twoheadrightarrow U\Big(\g\otimes \big(\C[t,t^{-1}]/(t-1)^m\big)\Big)$ is $0$ and hence the degree of $\mathcal{M}$ is at least $m$. Therefore, $\h^{m_0} \mathcal{M}\in \bK_m$ if and only if the degree of $\mathcal{M}$ is at least $m-m_0$.

We claim that the elements $\ov{\h^{m_0} \mathcal{M}}$, where $0\leq m_0\leq m$ and the degree of $\mathcal{M}$ equals $m-m_0$, form a $\C$-basis of $\bK_m/\bK_{m+1}$. In the proof of the surjectivity of $\Phi$, we have showed that these elements form a spanning set. It suffices to show that such elements are linearly independent. Suppose that we have a linear combination $\sum_i a_i\ov{\h^{m_i}\mathcal{M}_i}=0$ for some monomial $\mathcal{M}_i$ with degree $m-m_i$, $a_i\in \C$. Then we have 
\begin{align}
\sum_i a_i \h^{m_i}\mathcal{M}_i \in \bK_{m+1}.
\end{align}
Since the left-hand side has a total degree $m$, this identity forces that $\sum_i a_i \h^{m_i}\mathcal{M}_i=0$ and hence $a_i=0$ for all $i$. Hence, the claim is proved.
\end{proof}

\begin{theorem}
 \label{thm:PBW}
The ordered monomials of $\{\X_{\alpha,r},\H_{i,s} \mid \alpha\in \cR^+,i\in \I^0,r\in \N, s\in 2\N+1 \}$ form a basis of $\Y$. Hence $\Y$ is a flat $\N$-graded deformation of $U\big(\g[u]^{\check\omega}\big)$ over $\C[\hbar]$, and in particular isomorphic to $U\big(\g[u]^{\check\omega}\big)[\hbar]$ as a graded $\C[\hbar]$-module. Here $U\big(\g[u]^{\check\omega}\big)[\h]=U\big(\g[u]^{\check\omega}\big)\otimes_{\C}\C[\h]$ is given the natural $\C[\h]$-module structure.
\end{theorem}

\begin{proof}
These monomials form a spanning set by Proposition~\ref{prop:PBWbasis} and they are linearly independent by the proof of Theorem~\ref{thm:qiso}. The second part follows from this basis result and the comparison of the presentations of $\Y$ in Definition~ \ref{def:YN} and of $U\big(\g[u]^{\check\omega}\big)$ in Proposition~ \ref{prop:crel}.
\end{proof}


We define a filtration on $\Y$ by setting $\widetilde\deg \, \h=0,\widetilde\deg \, \H_{i,s}=s, \widetilde\deg \, \X_{i,m}=m$, and denote by $\widetilde{\text{gr}}\,\Y$ the associated graded algebra of $\Y$. Let $\widetilde{\H}_{i,s}, \widetilde{\X}_{i,m}$ denote the images of $\H_{i,s},\X_{i,m}$ in  $\widetilde{\text{gr}}\,\Y$. 

\begin{corollary}
 \label{cor:PBW2}
Let $\Y$ be a twisted Yangian of split type (excluding type $G_2$). There is a graded $\C[\h]$-algebra isomorphism  
\[
\iota: U\big(\g[u]^{\check\omega}\big)[\h] \stackrel{\cong}{\longrightarrow} \widetilde{\mathrm{gr}}\,\Y
\]
which sends $2 t_{i,s}\mapsto \widetilde{\H}_{i,s}, x_{i,m}\mapsto  \widetilde{\X}_{i,m}$. 
\end{corollary}


\subsection{A reduced presentation for $\Y$}

We give an alternative and simpler presentation of the twisted Yangian $\Y$ defined in Definition \ref{def:YN}. This new presentation can be useful in practice as it is much easier to verify. 

\begin{theorem}
\label{thm:Yreduced}
    The twisted Yangian $\Y$ has the following  presentation as a $\C[\h]$-algebra with generators $\H_{i,s},\X_{j,r}$, for $s\in 2\N+1,r\in\N,i,j\in \mathbb{I}^0$, subject to the relations \eqref{ty0}--\eqref{ty2} 
and the finite type Serre relations \eqref{fSerre0}--\eqref{fSerre2} below:
\begin{align}
\label{fSerre0}
 [b_{i,0},b_{j,0}] 
&=0 \qquad\qquad\qquad\;\; (c_{ij}=0),
\\
\label{fSerre1}
 \big[b_{i,0},[b_{i,0},b_{j,0}]\big] 
&=- b_{j,0} \qquad\qquad\; (c_{ij}=-1),
\\
\label{fSerre2}
 \Big[b_{i,0},\big[b_{i,0},[b_{i,0},b_{j,0}]\big]\Big] 
 &= - 4 [b_{i,0},b_{j,0}] \qquad (c_{ij}=-2).
\end{align}
\end{theorem}

\begin{proof}
    Let us denote by ${}'\Y$ the $\C[\h]$-algebra with presentation given in the statement of the theorem. As ${}'\Y$ and $\Y$ have the same generators and ${}'\Y$ has fewer relations than $\Y$, we clearly have a surjective homomorphism $f: {}'\Y \rightarrow \Y$. Taking the associated graded as in Corollary \ref{cor:PBW2}, we obtain a surjective homomorphism $\widetilde{f}  : \widetilde{\text{gr}}\,{}'\Y \rightarrow \widetilde{\text{gr}}\,\Y$. 

    On the other hand, we have an algebra isomorphism $\iota: U\big(\g[u]^{\check\omega}\big) \rightarrow \widetilde{\text{gr}}\,\Y$, which matches the generators. By the presentation of $U\big(\g[u]^{\check\omega}\big)$ in Proposition~\ref{prop:reduced relation}, we also have a natural homomorphism $U\big(\g[u]^{\check\omega}\big) \rightarrow \widetilde{\text{gr}}\,{}'\Y$, which match the generators. By composition of $\iota^{-1}$ with this, we obtain a homomorphism $\widetilde{\text{gr}}\,\Y \rightarrow \widetilde{\text{gr}}\, {}'\Y$; this homomorphism is clearly the inverse to $\widetilde{f}$ as seen on generators. Hence, we have proved that $\widetilde{f} : \widetilde{\text{gr}}\, {}'\Y \rightarrow \widetilde{\text{gr}}\,\Y$ is an isomorphism, and so is $f: {}'\Y \rightarrow \Y$. 
\end{proof}

\section{Verification of the relations} \label{sec:verify}

 In this section, we show that the defining relations \eqref{ty0}--\eqref{ty4} of $\Y$ hold for the corresponding generators in $\Gr_{\mathbf{K}}\tUi$ under the map $\Phi:\Y\rightarrow \Gr_{\mathbf{K}}\tUi$ in \eqref{Phimap}. The verification of the relation \eqref{ty6} is postponed to Section~ \ref{sec:gen}. This implies that the map $\Phi:\Y\rightarrow \Gr_{\mathbf{K}}\tUi$ is an algebra homomorphism. 

\subsection{Relation \eqref{ty0}} We claim that the following relation holds in $\Gr_{\mathbf{K}}\tUi$:
\begin{align*}
[\oH_{i,s,1},\oH_{j,r,1}]=0.
\end{align*}
Indeed, taking a summation with respect to $k,l$ in \eqref{idef1}, we obtain that 
$[\H_{i,s,1},\H_{j,r,1}]=0.$
Hence, in $\bK_{r+s}/ \bK_{r+s+1}$, the desired relation follows.

\subsection{Relation \eqref{ty5}}

We will often write $a=b+O(\h^r)$, which means that $a-b$ is an element in $\bK_{r}$. 

We verify the relation \eqref{ty5} in $\Gr_{\mathbf{K}}\tUi$ using the first few cases of \eqref{idef2}.

\begin{proposition}
The following relation holds in $\Gr_{\mathbf{K}}\tUi$:
\begin{equation}\label{deg5''}
[\ov{\H}_{i,1,1},\ov{\X}_{j,r,l}]=2 c_{ij}\ov{\X}_{j,r+1,l}.
\end{equation}
\end{proposition}

\begin{proof}
Setting $k=1,2$ in \eqref{idef2}, we have
\begin{align*}
[H_{i,1},B_{j,l}]&= [ c_{ij}]_i B_{j,l+1}- [ c_{ij}]_i B_{j,l-1},
\\
[H_{i,2},B_{j,l}]&=\frac{[2c_{ij}]_i}{2}B_{j,l+2}-\frac{[2c_{ij}]_i}{2}B_{j,l-2}.
\end{align*}

Note by definition that $\Theta_{i,1}=H_{i,1}$ and $\Theta_{i,2}=H_{i,2}+\frac{v_i-v_i^{-1}}{2!}H_{i,1}^2$.
Since $\H_{i,1,1}=\Theta_{i,2}-\Theta_{i,1}$, we take the difference of the above two identities and obtain that 
\begin{equation}
[\H_{i,1,1},\X_{j,r,l}]
= c_{ij}\big(\X_{j,r+1,l+1}+\X_{j,r+1,l-2}\big)+O(\h^{r+2}).
\end{equation}
Thus, by Corollary \ref{lemma3}, the desired relation follows in $\bK_{r+1}/ \bK_{r+2}$.
\end{proof}

\subsection{Relation \eqref{ty1}} 

We verify the relation \eqref{ty2} in $\Gr_{\mathbf{K}}\tUi$  using \eqref{idef2'}, which is known by Proposition \ref{prop:equiv} to be equivalent to \eqref{idef2}.

\begin{proposition}
The following relation holds in $\Gr_{\mathbf{K}}\tUi$:
\begin{align}\label{deg2'}
\begin{split}
&[\ov{\X}_{j,r,1},\ov{\H}_{i,s+2,1}]-[\ov{\X}_{j,r+2,1},\ov{\H}_{i,s,1}]
\\=&-\h d_i c_{ij}( \ov{\X}_{j,r+1,1} \ov{\H}_{i,s,1} +  \ov{\H}_{i,s,1} \ov{\X}_{j,r+1,1})+ \frac{\h^2 d_i^2 c^2_{ij}}{4}[ \ov{\X}_{j,r,1}, \ov{\H}_{i,s,1}].
\end{split}
\end{align}
\end{proposition}

\begin{proof}
Taking a summation with respect to $k,l$ in \eqref{idef2'}, we have
\begin{align*}
[\X_{j,r,l},\H_{i,s,k+2}]+[\X_{j,r,l},\H_{i,s,k} ]&-v_i^{c_{ij}}\X_{j,r,l-1}\H_{i,s,k+1} -v_i^{-c_{ij}}\X_{j,r,l+1}\H_{i,s,k+1}\\
&+v_i^{c_{ij}}\H_{i,s,k+1}\X_{j,r,l+1}+v_i^{-c_{ij}}\H_{i,s,k+1}\X_{j,r,l-1} =0.
\end{align*}
Hence, using Lemma~\ref{lemma1}, we have 
\begin{align*}
&[\X_{j,r,l},\H_{i,s+1,k+1}]-[\X_{j,r+1,l},\H_{i,s,k+1}]+[\X_{j,r+1,l-1},\H_{i,s,k}]-[\X_{j,r,l-1},\H_{i,s+1,k}]\\
=&\frac{\h d_i c_{ij}}{2}(\X_{j,r,l-1}\H_{i,s,k+1} - \X_{j,r,l+1}\H_{i,s,k+1} - \H_{i,s,k+1}\X_{j,r,l+1} + \H_{i,s,k+1}\X_{j,r,l-1})\\
&+\frac{\h^2 d_i^2 c^2_{ij}}{8}(\X_{j,r,l-1}\H_{i,s,k+1}-\H_{i,s,k+1}\X_{j,r,l-1}+\X_{j,r,l+1}\H_{i,s,k+1}-\H_{i,s,k+1}\X_{j,r,l+1})
\\
&+O(\h^{r+s+3}).
\end{align*}
We rewrite the above formula as  
\begin{align*}
&[\X_{j,r+1,l-1},\H_{i,s+1,k}]+[\X_{j,r,l-1},\H_{i,s+2,k}]-[\X_{j,r+2,l-1},\H_{i,s,k}]-[\X_{j,r+1,l-1},\H_{i,s+1,k}]\\
=&-\frac{\h d_i c_{ij}}{2}(\X_{j,r+1,l}\H_{i,s,k+1} + \X_{j,r+1,l-1}\H_{i,s,k+1} + \H_{i,s,k+1}\X_{j,r+1,l} + \H_{i,s,k+1}\X_{j,r+1,l-1})\\
&+\frac{\h^2 d_i^2 c^2_{ij}}{8}(\X_{j,r,l-1}\H_{i,s,k+1}-\H_{i,s,k+1}\X_{j,r,l-1}+\X_{j,r,l+1}\H_{i,s,k+1}-\H_{i,s,k+1}\X_{j,r,l+1})\\
&+O(\h^{r+s+3}).
\end{align*}
Note that the first term and the last term on the left-hand side are canceled. By Corollary~ \ref{lemma3}, in $\bK_{r+s+2}/\bK_{r+s+3}$, the desired relation \eqref{deg2'} follows.
\end{proof}

\subsection{Relation \eqref{ty2}}
We verify the relation \eqref{ty2} in $\Gr_{\mathbf{K}}\tUi$  using \eqref{idef3'}, which is known by Proposition \ref{prop:equiv} to be equivalent to \eqref{idef3}.

\begin{proposition}
The following relation holds in $\Gr_{\mathbf{K}}\tUi$:
\begin{align}
&[\oX_{i,s,1},\oX_{j,r+1,1}]-[\oX_{i,s+1,1},\oX_{j,r,1}]+\frac{\h d_i c_{ij}}{2}(\oX_{j,r,1}\oX_{i,s,1}+\oX_{i,s}\oX_{j,r,1})
= 2\delta_{ij}(-1)^{s} \oH_{i,r+s+1,1} .\label{deg3'}
\end{align}
\end{proposition}

\begin{proof}
Taking a summation with respect to $m$ in \eqref{idef3'}, we have
\begin{equation}\label{degen1}
[\X_{i,0,k},\X_{j,r,k+3}]_{v_i^{-c_{ij}}}-v_i^{-c_{ij}}[\X_{i,0,k+1},\X_{j,r,k+2}]_{v_i^{c_{ij}}}=\delta_{ij}v_i^{-2}(\H_{i,r,3} -\H_{i,r,1}).
\end{equation}
Replacing $k$ by $k-1$ in \eqref{degen1}, we have
\begin{equation}\label{degen2}
[\X_{i,0,k-1},\X_{j,r,k+2}]_{v_i^{-c_{ij}}}-v_i^{-c_{ij}}[\X_{i,0,k},\X_{j,r,k+1}]_{v_i^{c_{ij}}}=\delta_{ij}v_i^{-2}(\H_{i,r,3} -\H_{i,r,1}).
\end{equation}
Taking the difference between \eqref{degen2} and \eqref{degen1} and then using Lemma \ref{lemma1}, we have
\begin{equation}\label{degen3}
\begin{aligned}
&\quad [\X_{i,1,k-1},\X_{j,r,k+2}]_{v_i^{-c_{ij}}}-v_i^{-c_{ij}}[\X_{i,1,k},\X_{j,r,k+1}]_{v_i^{c_{ij}}}
\\&+[\X_{i,0,k-1},\X_{j,r+1,k+2}]_{v_i^{-c_{ij}}}-v_i^{-c_{ij}}[\X_{i,0,k},\X_{j,r+1,k+1}]_{v_i^{c_{ij}}}=O(\h^{r+3}),
\end{aligned}
\end{equation}
where
\begin{align*}
O(\h^{r+3})=[\X_{i,1,k-1},\X_{j,r+1,k+2}]_{v_i^{-c_{ij}}}-v_i^{-c_{ij}}[\X_{i,1,k},\X_{j,r+1,k+1}]_{v_i^{c_{ij}}}.
\end{align*}
Note that last two terms on the LHS of \eqref{degen3} is the same as the LHS of \eqref{degen1}, and hence we can replace it by RHS of \eqref{degen1}. Then we obtain
\begin{equation}
\begin{aligned}
&[\X_{i,1,k-1},\X_{j,r,k+2}]_{v_i^{-c_{ij}}}-v_i^{-c_{ij}}[\X_{i,1,k},\X_{j,r,k+1}]_{v_i^{c_{ij}}}
= \delta_{ij}v_i^{-2}(-\H_{i,r+1,3} +\H_{i,r+1,1}) +O(\h^{r+3}).
\end{aligned}
\end{equation}
Inductively, we obtain the following identity for $s\geq 0$
\begin{equation}
\begin{aligned}
&[\X_{i,s,k-s},\X_{j,r,k+3-s}]_{v_i^{-c_{ij}}}-v_i^{-c_{ij}}[\X_{i,s,k+1-s},\X_{j,r,k+2-s}]_{v_i^{c_{ij}}}\\
&= \delta_{ij}v_i^{-2} (-1)^{s} (\H_{i,r+s,3}- \H_{i,r+s,1})+O(\h^{r+s+2}).
\end{aligned}
\end{equation}
We rewrite the above identity as follows : 
\begin{align}\notag
&[\X_{i,s,k-s},\X_{j,r+1,k+2-s}]-[\X_{i,s+1,k-s},\X_{j,r,k+2-s}]
\\
&\quad +\frac{\h d_i c_{ij}}{2}(\X_{j,r,k+3-s}\X_{i,s,k-s}+\X_{i,s,k+1-s}\X_{j,r,k+2-s})\\\notag
&= \delta_{ij}(-1)^{s}  (\H_{i,r+s+1,2}+ \H_{i,r+s+1,1})+ O(\h^{r+s+2}).
\end{align}
By Corollary~\ref{lemma3}, in $\bK_{r+s+1}/\bK_{r+s+2}$, the desired relation follows.
\end{proof}


\subsection{Serre relation \eqref{ty4}}\label{Serre}

Let $i,j\in \I^0$ such that $c_{ij}=-1$. We derive the Serre relation \eqref{ty4} from \eqref{idef4}.
We shall use the following shorthand notation:
\begin{align}
S(x,y,z)=\big[x, [y,z]\big] +\big[y, [x,z]\big],
\end{align}
for $x,y,z\in \Gr_{\mathbf{K}}\tUi$.
 \begin{proposition}
The following relation holds in $\Gr_{\mathbf{K}}\tUi$:
 \begin{align}\notag
&\quad S(\oX_{i,m,1},\oX_{i,s,1},\oX_{j,r,1})
\\ \notag
&=2 \delta_{s+m,ev} (-1)^{m+1} \left( \frac{\h d_i}{2}\right)^{s+m} \oX_{j,r,1}
+(-1)^{m+1} \sum_{t=0}^{\left\lfloor \frac{s+m-1}{2} \right\rfloor} \left( \frac{\h d_i}{2}\right)^{2t} [\oX_{j,r+1,1}, \oH_{i,s+m-2t-1,1}]
\\ 
&\quad +2(-1)^{m+1} \sum_{t=0}^{\left\lfloor \frac{s+m-1}{2} \right\rfloor} \left( \frac{\h d_i}{2}\right)^{2t+1}   ( \oH_{i,s+m-2t-1,1} \oX_{j,r,1}+\oX_{j,r,1} \oH_{i,s+m-2t-1,1}). \label{eq:Serre7}
\end{align}
\end{proposition}

\begin{proof}
On one hand, the RHS of \eqref{idef4} only depends on $k_2-k_1$, and it implies that the formula of $S(\X_{i,m,k_1},\X_{i,s,k_2},\X_{j,r,l})$ only depends on $k_2-k_1$. Then we have
\begin{align}\notag
&S(\X_{i,m,k_1},  \X_{i,s+1 ,k_2},\X_{j,r,l})+S(\X_{i,m,k_1},  \X_{i,s,k_2},\X_{j,r,l})+S(\X_{i,m-1,k_1},  \X_{i,s +1,k_2},\X_{j,r,l})
\\\label{eq:Serre3}
=&S(\X_{i,m-1,k_1+1},  \X_{i,s ,k_2+1},\X_{j,r,l})-S(\X_{i,m-1,k_1},  \X_{i,s,k_2},\X_{j,r,l})=0.
\end{align}
Since $S(\X_{i,m,k_1},  \X_{i,s+1,k_2},\X_{j,r,l})\in O(\h^{m+s+r+1})$, we rewrite \eqref{eq:Serre3} as
\begin{align}
S(\X_{i,m,k_1},  \X_{i,s ,k_2},\X_{j,r,l})+S(\X_{i,m-1,k_1},  \X_{i,s +1,k_2},\X_{j,r,l})= O(\h^{m+s+r+1}).
\label{eq:Serre4}
\end{align}
It follows that $S(\X_{i,m },  \X_{i,s },\X_{j,r })=0$ in $\Gr_\mathbf{K} \widetilde{\mathbf{U}}^\imath$ if $m+s$ is odd.

On the other hand, taking a summation with respect to $l$ for \eqref{idef4}, we obtain 
\begin{align}
   \label{eq:Serre1}
\begin{split}
&\quad S(\X_{i,0,k_1},\X_{i,0,k_2},\X_{j,r,l})\\
&=-\sum_{p=0}^{\bigl\lfloor\frac{k_2-k_1-1}{2}\bigr\rfloor } (v_i^{2p+1}+v_i^{-2p-1}) [\Theta_{i,k_2-k_1-2p-1},\X_{j,r,l-1}]_{v_i^{-2}} \\
&\quad-\sum_{p=1}^{\bigl\lfloor\frac{k_2-k_1 }{2}\bigr\rfloor-1} (v_i^{2p}+v_i^{-2p})[\X_{j,r,l},\Theta_{i,k_2-k_1-2p}]_{v_i^{-2}}
\\
&\quad-(v_i^{k_2-k_1-1}+v_i^{k_1-k_2-1}) \X_{j,r,l}  -[\X_{j,r,l},\Theta_{i,k_2-k_1}]_{v_i^{-2}}
\end{split}
\end{align}
Suppose that $k_2-k_1\gg 0$ and it is even. We calculate the difference $\eqref{eq:Serre1}_{k_1,k_2+2}-\eqref{eq:Serre1}_{k_1,k_2}$ as follows:
\begin{align*}
& S(\X_{i,0,k_1}, \X_{i,1,k_2}+\X_{i,1,k_2+1},\X_{j,r,l})
\\
&=- (v_i+v_i^{-1}) [\Theta_{i,k_2-k_1+1},\X_{j,r,l-1}]_{v_i^{-2}} \\
&\quad-(v_i^{2}+v_i^{-2})[\X_{j,r,l},\Theta_{i,k_2-k_1}]_{v_i^{-2}} -[\X_{j,r,l},\H_{i,1,k_2-k_1}+\H_{i,1,k_2-k_1+1}]_{v_i^{-2}}\\
&\quad -\sum_{p=0}^{\bigl\lfloor\frac{k_2-k_1-1}{2}\bigr\rfloor } (v_i^2-1)(v_i^{2p+1}-v_i^{-2p-3}) [\Theta_{i,k_2-k_1-2p-1},\X_{j,r,l-1}]_{v_i^{-2}} \\
&\quad-\sum_{p=1}^{\bigl\lfloor\frac{k_2-k_1 }{2}\bigr\rfloor-1} (v_i^2-1)(v_i^{2p}-v_i^{-2p-2})[\X_{j,r,l},\Theta_{i,k_2-k_1-2p}]_{v^{-2}}\\
&\quad - (v_i^2-1) (v_i^{k_2-k_1-1}-v_i^{k_1-k_2-3}) \X_{j,r,l}
\\
&= -(v_i^{2}+v_i^{-2})[\X_{j,r+1,l-1},\H_{i,0,k_2-k_1}]_{v_i^{-2}}
\\
&\quad-4(v_i-1)(\H_{i,0,k_2-k_1} \X_{j,r,l-1}+\X_{j,r,l-1} \H_{i,0,k_2-k_1})
\\
&\quad +O(\h^{r+2}).
\end{align*}
By a similar computation, we have
\begin{align*}
 S(\X_{i,0,k_1}, \X_{i,2,k_2}+2\X_{i,2,k_2+1}+\X_{i,2,k_2+2},\X_{j,r,l})
&=-4[\X_{j,r+1,l-1},\H_{i,1,k_2-k_1}]_{v_i^{-2}}
\\
&\quad-4\h d_i(\H_{i,1,k_2-k_1} \X_{j,r,l-1}+\X_{j,r,l-1} \H_{i,1,k_2-k_1})
\\
&\quad - 2\h^2 \X_{j,r,l} +O(\h^{r+3}).
\end{align*}
Inductively, we have
\begin{align}\label{eq:Serre2}
\begin{split}
&\quad S\Big(\X_{i,0,k_1}, \sum_{t=0}^s \binom{s}{t}\X_{i,s,k_2+t},\X_{j,r,l}\Big)
\\
&=- \sum_{t=0}^{\left\lfloor \frac{s-1}{2} \right\rfloor} 2^{s-2t} (\h d_i)^{2t} [\X_{j,r+1,l-1}, \H_{i,s-2t-1,k_2-k_1}] 
\\
&\quad- \sum_{t=0}^{\left\lfloor \frac{s-1}{2} \right\rfloor} 2^{s-2t} (\h d_i)^{2t+1}  ( \H_{i,s-2t-1,k_2-k_1} \X_{j,r,l-1}+\X_{j,r,l-1} \H_{i,s-2t-1,k_2-k_1})
\\
&\quad -2 \delta_{s,ev} (\h d_i)^{s} \X_{j,r,l}+O(\h^{r+s+1}).
\end{split}
\end{align}
Using \eqref{eq:Serre2} and \eqref{eq:Serre4}, we obtain the desired relation \eqref{eq:Serre7}.
\end{proof}

We remark that the RHS of \eqref{eq:Serre7} is nonzero if and only if $s+m$ is even.

\section{Degeneration via generating functions}
\label{sec:gen}

In this section, we prove that the Serre relation \eqref{ty6} for $c_{ij}=-2$ in $\Y$ holds for the corresponding generators in $\Gr_{\mathbf{K}}\tUi$ in generating function form via degeneration. We find it illuminating to first give a quick new derivation of the counterparts of the relations \eqref{ty0}--\eqref{ty4} in generating function form for $\Y$ hold in $\Gr_{\mathbf{K}}\tUi$.  

\subsection{New notations}
Recall the definition~\eqref{eq:Genfun}. Define
\begin{align}
\label{ThetaBz}
\bTH_{i}^+(z)= \sum_{k>0} \TH_{i,k} z^k,\qquad \bB_i^+(z)=\sum_{k> 0} B_{i,k} z^k.
\end{align}
Recall that 
\[
\frac{u^k}{(1+u)^k}=\sum_{s\geq 0} (-1)^s \binom{s+k-1}{s}u^{s+k}. 
\]
Introduce a new variable $u$ such that $z=\frac{u}{1+u}.$ By the definition \eqref{def:HX}, the generating functions in \eqref{ThetaBz} can be rewritten in terms of $u$ as
\begin{align}
 \bTH^+_{i}\big(\frac{u}{1+u}\big)=\sum_{k\geq 0} \H_{i,k,1}u^{k+1},
 \qquad  
 \bB_{i}^+\big(\frac{u}{1+u}\big)=\sum_{k\geq 0} \X_{i,k,1}u^{k+1}.
\end{align}
We define generating functions $\bh_i(u),\bb_i(u)$ such that \[
\bh_i(u^{-1})=\bTH^+_{i}\big(\frac{u}{1+u}\big), 
\qquad
\bb_i(u^{-1})=\bB_{i}^+\big(\frac{u}{1+u}\big). 
\]
We will also use notations $\bh_i(u),\bb_i(u)$ for their images in $\Gr_{\mathbf{K}}\tUi$.

\subsection{Relations \eqref{ty0}--\eqref{ty2}} 
\begin{proposition}\label{prop:Gen}
In $\Gr_{\mathbf{K}}\tUi$, we have
\begin{align}
&[\bh_i(u_1),\bh_j(u_2)]=0,
\\\label{iGen2}
&(u_1-u_2)[\bb_i(u_1),\bb_i(u_2)]=-\h d_i(\bb_i(u_1)-\bb_i(u_2))^2 
- \frac{u_1-u_2}{ u_1+u_2} \big(\bh_i(u_1)-\bh_i(u_2)\big),
\\\notag
&(u_1-u_2)[\bb_i(u_1),\bb_j(u_2)]=\frac{d_i c_{ij}\h}{2}\big\{\bb_i(u_1),\bb_j(u_2)\big\}+ [b_{i,0},\bb_j(u_2)]-[\bb_i(u_1),b_{j,0}]
\\\label{iGen2'}
&\qquad\qquad\qquad\qquad \qquad\quad- \delta_{ij}\frac{u_1-u_2}{ u_1+u_2} \big(\bh_i(u_1)-\bh_i(u_2)\big)-\delta_{ij}\big(\bh_i(u_1)+\bh_i(u_2)\big),
\\\label{iGen3}
&(u_1^2-u_2^2)[\bh_i(u_1),\bb_j(u_2)]=d_i c_{ij}\h u_2\{\bh_i(u_1),\bb_j(u_2)\}+\frac{\h^2}{4}d_i^2 c_{ij}^2[\bh_i(u_1),\bb_j(u_2)]
\\\notag&\quad +2 c_{ij} u_2 \bb_j(u_2)-[\bh_i(u_1),\ov{b}_{j,1,1}]-d_i c_{ij}\h\big\{\bh_i(u_1),\ov{b}_{j,0,1}\big\}-2 c_{ij} \ov{b}_{j,0,1}-u_2[\bh_i(u_1),\ov{b}_{j,0,1}].
\end{align}
\end{proposition}

 
\begin{proof}
The first identity can be easily obtained from \eqref{iDRG1}. 

We prove \eqref{iGen2}. Taking the components for all $z^k w^l, k,l> 0$ in the defining relation \eqref{iDRG3b}, we have
\begin{align} 
\notag
&(v_i^2 z-w)  \bB_i^+(z) \bB_i^+(w) + (v_i^2 w-z)  \bB_i^+(w) \bB_i^+(z)-z[\bB_i^+(w),B_{i,0}]_{v_i^2}-w[\bB_i^+(z),B_{i,0}]_{v_i^2}
\\
&= v_i^{-2}\frac{z-w}{1-zw} \big( \bTH_{i}^+(w)-\bTH_{i}^+(z)\big)+v_i^{-2} w\bTH_{i}^+(w)+v_i^{-2} z\bTH_{i}^+(z).
\label{eq:Gen2}
\end{align}
Setting $z=w$ in \eqref{eq:Gen2}, we obtain $(v_i^2-1)\bB_i^+(z)^2-[\bB_i^+(z),B_{i,0}]_{v_i^2}=v_i^{-2}\bTH_{i}^+(z)$. Then we rewrite \eqref{eq:Gen2} as
\begin{align}\label{eq:Gen3}
\begin{split}
&(v_i^2 z-w)  \bB_i^+(z) \bB_i^+(w) + (v_i^2 w-z)  \bB_i^+(w) \bB_i^+(z) 
\\
&= v_i^{-2}zw\frac{z-w}{1-zw} \big( \bTH_{i}^+(w)-\bTH_{i}^+(z)\big)
+(v_i^2-1) \big(z\bB_i^+(w)^2+w\bB_i^+(z)^2\big).
\end{split}
\end{align}
Substitute $z=\frac{u_1^{-1}}{1+u_1^{-1}}=\frac{1}{1+u_1},w=\frac{1}{1+u_2}$, and then we obtain
\begin{align}\label{eq:Gen4}
\begin{split}
&(u_1-u_2)[\bb_i(u_1),\bb_i(u_2)]=-\frac{u_1-u_2}{\boxed{\blue{u_1 u_2}+} u_1+u_2}\Big(\bh_i(u_1)-\bh_i(u_2)\Big) 
\\
&-\h d_i \Big((1 \boxed{+\blue{u_2}})\big(\bb_i(u_1)^2-\bb_i(u_1)\bb_i(u_2)\big)+(1 \boxed{+\blue{u_1}})\big(\bb_i(u_2)^2-\bb_i(u_2)\bb_i(u_1)\big)\Big).
\end{split}
\end{align}
As the three boxed blue terms above lie in the higher degree components, we can simplify the above identity in the associated graded $\Gr_{\mathbf{K}}\tUi$ by removing the boxed blue terms, and obtain \eqref{iGen2}. The relation \eqref{iGen2'} can be proved similarly. 

We next prove \eqref{iGen3}. Taking the components for $z^k w^l, k,l> 0$ in \eqref{iDRG2}, we obtain the following identity
\begin{align}\notag
&(w -v_i^{-c_{ij}}z ) (1 -v_i^{c_{ij}} zw) \B^+_j(w)  \bTH_i^+(z)
-(w -v_i^{c_{ij}}z ) (1 -v_i^{-c_{ij}}zw) \bTH_i^+(z) \B_j^+(w)
\\\notag
&= w \big[\bTH_i^+(z), zw B_{j,0} +zB_{j,-1} -B_{j,0}-z^2  B_{j,0}\big]
\\\notag
&\quad +[c_{ij}]_i z(1-w^2) \B^+_j(w) -zw  [c_{ij}]_i \big(wB_{j,0}+B_{j,-1}\big)
\\\label{eq:Gen5}
&\quad - (v_i^{c_{ij}}-1)zw \big(w B_{j,0} +B_{j,-1}\big) \bTH_i^+(z)
+ (v_i^{-c_{ij}}-1) zw \bTH_i^+(z)\big(w B_{j,0} +B_{j,-1}\big) .
\end{align}
Substitute $z=\frac{1}{1+u_1},w=\frac{1}{1+u_2}$ in the above identity. Note that we have the following formulas
\begin{align*}
1-w^2 =2u_2 + o(\h), \qquad v_i^{c_{ij}}-1 &=\frac{d_i c_{ij}}{2}\h +o(\h),
\\
w-v_i^{\pm c_{ij}}z &= u_1-u_2 \mp \frac{d_i c_{ij}}{2}\h + o(\h),
\\
1-v_i^{\pm c_{ij}} zw &= u_1 +u_2 \mp \frac{d_ic_{ij}}{2} \h + o(\h),
\\
w B_{j,0} +B_{j,-1} &=2b_{j,0,1}+o(\h),
\\
zw B_{j,0} +zB_{j,-1} -B_{j,0}-z^2  B_{j,0} &= - u_2 B_{j,0} - b_{j,1,1}+o(\h),
\end{align*}
where $o(\h)$ represents those higher degree terms (with respect to $u_1,u_2,\h$).
Using the above formulas, we simplify \eqref{eq:Gen5} and obtain the desired relation \eqref{iGen3}.
\end{proof}

\begin{corollary}
Relations \eqref{ty0}--\eqref{ty2} hold in $\Gr_{\mathbf{K}}\tUi$.
\end{corollary}

\begin{proof}
Note that the images of $h_i(u)$ and $b_j(u)$ under the map $\Phi:\Y \rightarrow \Gr_{\bK}\tUi$ are $1+\h d_i\bh_i(u)$ and $\h d_j\bb_j(u)$, respectively. Then clearly the formulas in Proposition~\ref{prop:Gen} imply that relations \eqref{ty0}--\eqref{ty2} hold in $\Gr_{\mathbf{K}}\tUi$.
\end{proof}
\subsection{Serre type relation for $c_{ij}=-1$}

Denote
\begin{align*}
\bS(w_1,w_2|z)^+=\Sym_{w_1,w_2}\Big\{&\B^+_{i}(w_1)\B^+_{i}(w_2)\B^+_{j}(z)-
\\
&-[2]_{i}\B^+_{i}(w_1)\B^+_{j}(z)\B^+_{i}(w_2)+\B^+_{j}(z)\B^+_{i}(w_1)\B^+_{i}(w_2)\Big\}.
\end{align*}
We rewrite \eqref{iDRG4} using the function $\Delta(w_1w_2)$ as follows
\begin{align}\notag
&\Sym_{w_1,w_2}\big\{\B_{i}(w_1)\B_{i}(w_2)\B_{j}(z)
-[2]_{v_i}\B_{i}(w_1)\B_{j}(z)\B_{i}(w_2)+\B_{j}(z)\B_{i}(w_1)\B_{i}(w_2)\big\}
\\\notag
&= \Sym_{w_1,w_2}\bDel(w_1w_2)
\frac{[2]_{i} z(w_2^2- 1)w_2}{(1-v_i^2 w_2^2)(1-v_i^{-2}w_2^2)}[\bTH_i(w_2),\B_{j}(z)]_{v_i^{-2}} 
\\\label{iDRG4'} 
&\quad +\Sym_{w_1,w_2} \bDel(w_1w_2)\frac{w_2^4-1}{(1-v_i^2 w_2^2)(1-v_i^{-2}w_2^2)}[\B_{j}(z),\bTH_i(w_2)]_{v_i^{-2}} .
\end{align}
Taking the components for positive powers of $w_1,w_2,z$ in \eqref{iDRG4'}, we obtain the following formula
\begin{align}
\bS(w_1,w_2|z)^+ \notag
=&\Sym_{w_1,w_2}  \frac{[2]_{i} z w_1 w_2}{1-w_1 w_2}
\frac{(w_2^2-1)w_2}{(1-v_i^2 w_2^2)(1-v_i^{-2}w_2^2)}[\bTH_i^+(w_2),\B^+_{j}(z)]_{v_i^{-2}} 
\\\notag
&+\Sym_{w_1,w_2} \frac{w_1 w_2}{1-w_1 w_2} \frac{ w_2^4-1}{(1-v_i^2 w_2^2)(1-v_i^{-2}w_2^2)}[\B^+_{j}(z),\bTH^+_i(w_2)]_{v_i^{-2}} 
\\\notag
&+v_i^{-1}\Sym_{w_1,w_2}  \frac{[2]_{i} z w_1 w_2}{1-w_1 w_2}
\frac{(w_2^2-1)w_2}{(1-v_i^2 w_2^2)(1-v_i^{-2}w_2^2)} \B^+_{j}(z)
\\\label{gen9}
&+v_i^{-1}\Sym_{w_1,w_2} \frac{w_1 w_2}{1-w_1 w_2} \frac{ w_2^4-1}{(1-v_i^2 w_2^2)(1-v_i^{-2}w_2^2)} \B^+_{j}(z)
\\\notag
&+\Sym_{w_1,w_2}  \frac{[2]_{i} z w_1 w_2}{1-w_1 w_2}
\frac{(w_2^2-1)w_2}{(1-v_i^2 w_2^2)(1-v_i^{-2}w_2^2)}\Big([\bTH_i^+(w_2),B_{j,0} ]_{v_i^{-2}}+v_i^{-1} B_{j,0}\Big).
\end{align}
Substituting $w_1=\frac{1}{1+u_1},w_2=\frac{1}{1+u_2}, z=\frac{1}{1+t}$, we obtain the following relation in $\Gr_{\mathbf{K}}\tUi$
\begin{align}\notag
&\Sym_{u_1,u_2}\Big[\bb_{i}(u_1),\big[\bb_{i}(u_2), \bb_{j}(t) \big]\Big]
\\\notag
&= \frac{1}{u_1 + u_2}\Sym_{u_1,u_2} \Big\{\frac{4u_2(t-\h d_i)}{4u_2^2-d_i^2\h^2} \bh_i(u_2) \bb_j(t)- \frac{4u_2(t+\h d_i)}{4u_2^2-d_i^2\h^2} \bb_j(t)\bh_i(u_2) \Big\}
\\\label{gen10}
&\quad -\frac{1}{u_1 + u_2}\Sym_{u_1,u_2} \frac{8 u_2 }{4u_2^2-d_i^2\h^2}   \bb_j(t)  .
\end{align}
To obtain the component-wise formula for \eqref{gen10}, one needs to rewrite $\frac{1}{u_1+u_2},\frac{4u_2 }{4u_2^2-d_i^2\h^2}$ as $\frac{u_1^{-1}}{1+u_1^{-1}u_2},
\frac{u_2^{-1} }{1-\left(\frac{\h d_i}{2u_2}\right)^2}$, respectively; then expand them as power series. Then it is clear that component-wise formula for the above relation is exactly \eqref{ty4}.

\begin{remark}
The Serre relation \eqref{iDRG4} can be simplified as follows:
\begin{align}\notag
&\Sym_{w_1,w_2}\Big[\B_{i}(w_1),\big[\B_{i}(w_2),\B_{j}(z)\big]_{v_i}\Big]_{v_i^{-1}} 
\\
&= (1-v_i^{-2})\bDel(w_1w_2)\Big(\frac{ z}{v_i w_2-z  }-\frac{  z}{v_i w_1- z  } \Big) \big(\bTH_i(w_2)-\bTH_i(w_1)\big) \B_j(z).
\end{align}
We rewrite the above relation using $\bDel(w_1w_2)$ as 
\begin{align}\notag
&\Sym_{w_1,w_2}\Big[\B_{i}(w_1),\big[\B_{i}(w_2),\B_{j}(z)\big]_{v_i}\Big]_{v_i^{-1}} 
\\
&= - (1-v_i^{-2})\Sym_{w_1,w_2}\bDel(w_1w_2) \Big(\frac{  z}{v_i w_1^{-1}- z  }-\frac{ z}{v_i w_1-z  } \Big) \bTH_i(w_1)  \B_j(z).
\end{align}
Taking components for positive powers of $w_1,w_2,z$ in the above identity, we obtain 
\begin{align}\notag
&\Sym_{w_1,w_2}\Big[\B_{i}^+(w_1),\big[\B^+_{i}(w_2),\B^+_{j}(z)\big]_{v_i}\Big]_{v_i^{-1}} 
\\\notag
& =-(1-v_i^{-2})\Sym_{w_1,w_2}\frac{w_1 w_2}{1-w_1 w_2} \Big(\frac{ w_1 z}{v_i-w_1 z  }-\frac{ 1}{v_iw_1z^{-1}- 1} \Big) \bTH_i^+(w_1)  \B_j^+(z)
\\
&\quad +\text{ higher degree terms}.
\end{align}
Substituting $w_1=\frac{1}{1+u_1},w_2=\frac{1}{1+u_2}, z=\frac{1}{1+t}$, we obtain the following relation in $\Gr_{\mathbf{K}}\tUi$
\begin{align} 
\Sym_{u_1,u_2}\Big[\bb_{i}(u_1),\big[\bb_{i}(u_2), \bb_{j}(t) \big]\Big]
\label{gen15}
= \h d_i \frac{1}{u_1 + u_2}\Sym_{u_1,u_2} \frac{2u_1}{(t+\frac{\h d_i}{2})^2-u_1^2} \bh_i(u_1) \bb_j(t).
\end{align}
Writing $\frac{1}{u_1 + u_2}$ and $\frac{2u_1}{(t+\frac{\h d_i}{2})^2-u_1^2}$ as $\frac{u_2^{-1}}{1+u_1u_2^{-1}}$ and $\frac{-2u_1^{-1}}{1-(t+\frac{\h d_i}{2})^2u_1^{-2} }$ respectively and then expanding into power series, we arrive at the following component-wise formula for \eqref{gen15}:
\begin{align}
\Sym_{k_1,k_2}\Big[b_{i,k_1} ,\big[b_{i,k_2}, b_{j,r} \big] \Big]
\label{gen16}
=2\h d_i (-1)^{k_2+1}  \sum_{s\geq 0}\sum_{p=0}^{2s} \binom{2s}{p} \Big(\frac{\h d_i}{2}\Big)^{2s-p} h_{k_1+k_2-2s-1} b_{j,r+p},
\end{align}
where $h_{i,-1}$ is regarded as $(\h d_i)^{-1}$. 
One can show that \eqref{gen16} implies  \eqref{ty4}.  
\end{remark}

\subsection{Serre type relation for $c_{ij}=-2$}
We introduce the following notation 
\begin{align}
\bS[x_1,x_2,x_3|y]=\Sym_{x_1,x_2,x_3} \Big[x_1, \big[x_2 ,[x_3,y]_{v_i^2}\big]\Big]_{v_i^{-2}},
\end{align}
for elements $x_1,x_2,x_3,y$ in $\Ui$. By definition, $\bS[x_1,x_2,x_3|y]$ is symmetric with respect to $x_1,x_2,x_3$.

Let 
\begin{align*}
g(w_1,w_2,w_3)=(w_1+w_2+w_3)(w_1^{-1}+w_2^{-1}+w_3^{-1})-[3]^2_{i}.
\end{align*}
We take the components for positive powers of $w_1,w_2,w_3,z$ in \eqref{gen4} and regard $w_1=\frac{1}{1+u_1},w_2=\frac{1}{1+u_2},,w_3=\frac{1}{1+u_3}, z=\frac{1}{1+t}$. Then we obtain a Serre type relation involving $\bb_i(u_1),\bb_i(u_2),\bb_i(u_3),\bb_j(t)$ by expanding these rational functions as power series. We introduce a degree by setting $\deg u_i=\deg t=\deg \h=1, \deg b_{i,k}= \deg b_{j,k}=k$ on $\Gr_{\bK} \Ui[[u_1,u_2,u_3,t]]$, and then each of generating series becomes homogeneous with total degree $-1$. Modulo the higher degree terms in the Serre type relation, we obtain the following relation
\begin{align}
\label{gen11}
& \bS\Big[\B_i^+(w_1),\B_i^+(w_2),\B_i^+(w_3)|\B_j^+(z)\Big]
\\\notag
&= [3]_{i}\Sym_{w_1,w_2,w_3} \frac{ w_3 z^{-1} }{1-w_2 w_3 } ( w_2^2-1)g(w_1,w_2,w_3)^{-1}
\\\notag
&\qquad\quad \times\bigg(v_i[2]_{i}\big[\B^+_i(w_1),[\B^+_j(z),\bTH^+_i(w_2)]_{v_i^{-4}}\big]-\big[[ \B^+_j(z),\B^+_i(w_1)]_{v_i^{-2}},\bTH^+_i(w_2)\big]\bigg)
\\\notag
&\quad +\Sym_{w_1,w_2,w_3} \frac{ (w_1+w_2+w_2^{-1}) w_3 }{1-w_2 w_3 }( w_2^2-1) 
g(w_1,w_2,w_3)^{-1}
\\\notag
&\qquad\quad \times\bigg(-v_i[2]_{i}\big[[\bTH^+_i(w_2),\B^+_j(z)]_{v_i^{-4}}, \B^+_i(w_1)\big]+\big[\bTH_i^+(w_2),[\B^+_i(w_1), \B^+_j(z)]_{v_i^{-2}}\big]\bigg)
\\\notag
&\quad +v_i^{-1}[2]_i^2 [3]_{i}\Sym_{w_1,w_2,w_3} \frac{ w_3 z^{-1} }{1-w_2 w_3 } ( w_2^2-1)g(w_1,w_2,w_3)^{-1}\big[\B^+_i(w_1), \B^+_j(z) \big] 
\\\notag
&\quad -v_i^{-1}[2]_{i}^2 \Sym_{w_1,w_2,w_3} \frac{ (w_1+w_2+w_2^{-1}) w_3 }{1-w_2 w_3 }( w_2^2-1) 
g(w_1,w_2,w_3)^{-1} \big[ \B^+_j(z) , \B^+_i(w_1)\big] 
\\\notag
&\quad +\text{ higher degree terms}.
\end{align} 

Substituting $w_r=\frac{1}{1+u_r} , z=\frac{1}{1+t}$ for $r=1,2,3$, we obtain the following relation in $\Gr_{\mathbf{K}}\tUi$:
\begin{align}\notag
&\Sym_{u_1,u_2,u_3}\Big[\bb_{i}(u_1),\big[\bb_{i}(u_2),[\bb_i(u_3), \bb_{j}(t)] \big]\Big]
\\\notag
&= \Sym_{u_1,u_2,u_3} \frac{-u_2}{(u_2+u_3)(3u_2^2-3\h^2+u_1^2)}\times
\\\notag
&\quad \times \Big\{(9t+9\h d_i + 3u_1)\bb_{i}(u_1)\bb_{j}(t) \bh_i(u_2)+(9t-9\h d_i + 3u_1)\bh_i(u_2)\bb_{j}(t) \bb_{i}(u_1)
\\\notag
&\qquad -(3t-3\h d_i +u_1) \bh_i(u_2)\bb_{i}(u_1)\bb_{j}(t)-(3t+3\h d_i+u_1)\bb_{j}(t)\bb_{i}(u_1) \bh_i(u_2)
\\\notag
&\qquad-(6t-12\h d_i +2u_1) \bb_{i}(u_1)\bh_i(u_2)\bb_{j}(t)-(6t+12\h d_i +2u_1)\bb_{j}(t) \bh_i(u_2)\bb_{i}(u_1)
\\\label{gen12}
&\qquad + 24 \big[\bb_i(u_1),\bb_j(t)\big]\Big\}.
\end{align}
We rewrite the above formula as follows:
\begin{align}\notag
&\Sym_{u_1,u_2,u_3}\Big[\bb_{i}(u_1),\big[\bb_{i}(u_2),[\bb_i(u_3), \bb_{j}(t)] \big]\Big]
\\\notag
&= \Sym_{u_1,u_2,u_3} \frac{-u_2}{(u_2+u_3)(3u_2^2-3\h^2+u_1^2)}\times
\\\notag
&\quad \times \Big\{ (3t+u_1) \big[[\bb_i(u_1),\bb_j(t)],\bh_i(u_2)\big]+2(3t+u_1) \big[\bb_i(u_1),[\bb_j(t),\bh_i(t)\big] 
\\\notag
&\qquad +3 \h d_i \big[\{\bb_i(u_1),\bh_i(u_2)\},\bb_j(t) \big]+9 \h d_i \big[\bb_i(u_1),\{\bh_i(u_2),\bb_j(t) \}\big]
\\\label{gen13}
&\qquad + 24 \big[\bb_i(u_1),\bb_j(t)\big]\Big\}.
\end{align}
The component-wise formula for $\eqref{gen13}$ is given by
\begin{align}\notag
& \Sym_{k_1,k_2,k_3}\Big[b_{i,k_1} ,\big[b_{i,k_2},[b_{i,k_3}, b_{j,r} ] \big]\Big]
\\\notag
 &= \text{Cyc}_{k_1,k_2,k_3}  \sum_{s\geq 0}\sum_{p=0}^s(-1)^{k_3+p+1}\binom{s}{p} 3^{-p-1}\h^{2s-2p}\times
\\\notag
&\quad \times\Big\{ 3\big[[b_{i,k_1+2p},b_{j,r+1}],h_{i,k_2+k_3-2s-1}\big]
+\big[[b_{i,k_1+2p+1},b_{j,r}],h_{i,k_2+k_3-2s-1}\big] +
\\\notag
&\qquad +6\big[b_{i,k_1+2p},[b_{j,r+1},h_{i,k_2+k_3-2s-1}]\big]
+2\big[b_{i,k_1+2p+1},[b_{j,r},h_{i,k_2+k_3-2s-1}]\big]+
\\\notag
&\qquad +3 \h d_i \big[\big\{b_{i,k_1+2p},h_{i,k_2+k_3-2s-1}\big\},b_{j,r} \big]
+9 d_i  \h \big[b_{i,k_1+2p},\{h_{i,k_2+k_3-2s-1},b_{j,r} \}\big]\Big\}
\\\notag
&\qquad + 8  \text{Cyc}_{k_1,k_2,k_3} \delta_{k_2+k_3,even}
\sum_{p= 0}^{\frac{k_2+k_3}{2}}(-1)^{k_3+p+1} \binom{\frac{k_2+k_3}{2}}{p}3^{-p}\h^{k_2+k_3-2p}\big[b_{i,k_1+2p},b_{j,r}\big],
\\\label{gen14} 
\end{align}
where $h_{i,l}=0$ if $l<0$. Here $\text{Cyc}_{k_1,k_2,k_3}$ denotes the sum over cyclic permutations on the indices $(k_1,k_2,k_3)$.


\section{Twisted Yangian of type $G_2$}
\label{sec:G2}

In this section, we formulate the twisted Yangian of type $G_2$ in a reduced presentation using finite type Serre relations (similar to Theorem~\ref{thm:Yreduced}). 

\subsection{Relations through degeneration}
\label{sec:G2rel}

Let $\Ui$ be the affine $\imath$quantum group of type $G_2$ in Drinfeld presentation as given in Proposition~\ref{prop:Dr}. The filtration on $\Ui$ and the associated graded $\Gr_{\mathbf{K}}\tUi$ in \eqref{GrK} still make sense. 
 
The formulation of the relations \eqref{ty0}--\eqref{ty2} (i.e., all the relations other than Serre relations) still makes sense for $c_{ij}=-3$. The verification in Section \ref{sec:verify} of the relations \eqref{ty0}--\eqref{ty2} for the corresponding generators in $\Gr_{\mathbf{K}}\tUi$ under the map $\Phi$ in \eqref{Phimap} remains valid for $c_{ij}=-3$. 

One can obtain via degeneration from the finite type Serre relation for $\tUi$ (cf. \cite[Section~ 5]{Z22}) the following finite type Serre relations in $\Gr_{\mathbf{K}}\tUi$ without difficulty: for $r\geq 0$, 
\begin{align*}
\big[\ov{b}_{i,0,1},[\ov{b}_{i,0,1},\ov{b}_{j,r,1}]\big]  &=- \ov{b}_{j,r,1} \qquad (c_{ij}=-1),
\\
\Big[\ov{b}_{i,0,1},\big[\ov{b}_{i,0,1},[\ov{b}_{i,0,1},\ov{b}_{j,r,1}]\big]\Big] &=-4 [\ov{b}_{i,0,1},\ov{b}_{j,r,1}] \qquad (c_{ij}=-2),
\\
\Big[\ov{b}_{i,0,1},\Big[\ov{b}_{i,0,1},\big[\ov{b}_{i,0,1},
[\ov{b}_{i,0,1},\ov{b}_{j,r,1}]\big]\Big]\Big] &=
-10 \big[\ov{b}_{i,0,1},[\ov{b}_{i,0,1},\ov{b}_{j,r,1}]\big] 
   -9\ov{b}_{j,r,1},
\qquad (c_{ij}=-3).
\end{align*}
Only the formula for $c_{ij}=-3$ above is new, and  the other formulas for $c_{ij}=-1, -2$ were already obtained in the previous sections. 

However, the degeneration approach gets overwhelmingly complicated toward obtaining the general form of current Serre relations in $\Gr_{\mathbf{K}}\tUi$ for $c_{ij}=-3$.

\subsection{PBW basis for ${}'\Y$}

Inspired by the reduced presentations in Theorem~ \ref{thm:Yreduced} for twisted Yangians of all split types except $G_2$ defined in Definition~\ref{def:YN}, we define the twisted Yangian of type $G_2$ in a reduced presentation as follows.


\begin{definition}
\label{def:G2}
Let $\I^0=\{1,2\}$ with $c_{21}=-3$ and $c_{12}=-1$. 
The twisted Yangian of type $G_2$ is the $\C[\h]$-algebra ${}'\Y$ generated by $\H_{i,s},\X_{j,r}$, for $s\in 2\N+1,r\in\N,i,j\in \mathbb{I}^0$, subject to the relations \eqref{ty0}--\eqref{ty2} and the following finite Serre type relations: 
\begin{align}
    \label{eq:tyG2-1}
    \big[b_{1,0}, [b_{1,0},b_{2,r}]\big] &=-b_{2,r},
    \\
    \label{eq:tyG2-2}
    \bigg[ b_{2,0},\Big[b_{2,0},\big[b_{2,0},
[b_{2,0},b_{1,r}]\big]\Big] \bigg]&=-10 \big[b_{2,0},[b_{2,0},b_{1,r}]\big]-9 b_{1,r}.
\end{align}
\end{definition}
By setting $\hbar=0$ in the defining relations for ${}'\Y$ we recover a presentation for $U\big(\g[u]^{\check\omega}\big)$. To justify Definition~ \ref{def:G2}, we shall show that ${}'\Y$ has the same size as $U\big(\g[u]^{\check\omega}\big)$ and hence is a flat deformation of the latter. 

Define a filtration on ${}'\Y$ by setting 
\begin{align}
 \label{breveDeg}
    \breve{\deg}\, b_{i,m}=m+1, 
    \qquad
    \breve{\deg}\, h_{i,r}=r+1, 
\end{align}
and denote by $\breve{{}'\Y}$ the associated graded algebra. Denote by $\breve{h}_{i,r},\breve{b}_{j,m}$ the images of $h_{i,r},b_{j,m}$ in $\breve{{}'\Y}$. By \eqref{ty0}--\eqref{ty5} and \eqref{bbalt}--\eqref{bhalt} (which are equivalent to \eqref{ty1}--\eqref{ty2}), we have the following identities in $\breve{{}'\Y}$
\begin{align}
  \label{tyG-1}
[\breve{h}_{i,r},\breve{h}_{j,s}]&=0,
\\\label{tyG-2}
[\breve{b}_{i,k+1},\breve{b}_{j,l}]&=[\breve{b}_{i,k},\breve{b}_{j,l+1}], 
\\\label{tyG-3}
[\breve{b}_{i,k},\breve{b}_{i,l}]&=0, 
\\\label{tyG-4}
[\breve{h}_{i,r},\breve{b}_{j,s}]&=0.
\end{align}

\begin{lemma}
The following identities hold in $\breve{{}'\Y}$: for $k_1,k_2,k_3,k_4,r\in \N$,
\begin{align}
\label{eq:tyG2-Serre'}
\big[\breve{b}_{1,k_1}, [\breve{b}_{1,k_2},\breve{b}_{2,r}]\big] &=0,
\\
\label{eq:tyG2-Serre}
\bigg[\breve{b}_{2,k_1},\Big[ \breve{b}_{2,k_2},\big[\breve{b}_{2,k_3},[\breve{b}_{2,k_4},\breve{b}_{1,r}]\big]\Big] \bigg]& =0.
\end{align}
\end{lemma}

\begin{proof}
We shall prove the identity \eqref{eq:tyG2-Serre}; the other identity \eqref{eq:tyG2-Serre'} can be proved similarly. 

Let $\bS(k_1,k_2,k_3,k_4,r)$ denote the left-hand side of \eqref{eq:tyG2-Serre}. By \eqref{tyG-3}, $\breve{b}_{i,k_1},\ldots,\breve{b}_{i,k_4}$ commute with each other and hence $\bS(k_1,k_2,k_3,k_4,r)$ is symmetric with respect to the first four components.
By \eqref{tyG-2}, $[\breve{b}_{i,k},\breve{b}_{j,0}]=[\breve{b}_{i,0},\breve{b}_{j,k}]$. Hence, we have
\begin{align*}
\bS(k_1,k_2,k_3,k_4,r )
&=\bS(k_1,k_2,k_3,0,k_4+r )= \bS(k_1,k_2,0,k_3,k_4+r )\\
&= \bS(k_1,k_2,0,0,k_3+k_4+r )=\bS(k_1,0,0,k_2,k_3+k_4+r )\\
&=\bS(k_1,0,0,0,k_2+k_3+k_4+r )=\bS(0,0,0,k_1,k_2+k_3+k_4+r )\\
&=\bS(0,0,0,0,k_1+k_2+k_3+k_4+r)=0,
\end{align*}
where the last equality follows from \eqref{eq:tyG2-2}. 
\end{proof}

For $\alpha\in \cR^+$, we define elements $f_{\alpha,r},\X_{\alpha,r}$ formally in the same way as \eqref{def:bfalpha}.

\begin{theorem}
  \label{thm:PBWbasisG2}
The ordered monomials of $\{\X_{\alpha,r},\H_{i,s} \mid \alpha\in \cR^+,i\in \I^0,r\in \N, s\in 2\N+1\}$ (with respect to any fixed total ordering of the set) form a basis of ${}'\Y$.
\end{theorem}

\begin{proof}
Let $\mathfrak{n}^-$ be the subalgebra of $\g$ generated by $f_i,i\in\I^0$. Recall the filtration on ${}'\Y$ defined in terms of $\breve{\deg}$ \eqref{breveDeg} with the associated graded algebra $\breve{{}'\Y}$. 
By \eqref{tyG-1}--\eqref{tyG-4} and \eqref{eq:tyG2-Serre'}--\eqref{eq:tyG2-Serre}, $\breve{{}'\Y}$ is identified with (a quotient of) the current algebra $U(\mathfrak{n}^-[u])\otimes \C[\breve{h}_{i,s}|i\in \I^0, s\in 2\N+1]$ via $\breve{b}_{\alpha,r}\mapsto f_{\alpha}u^r,\breve{h}_{i,s}\mapsto \breve{h}_{i,s}$. By the PBW theorem for the current algebra, the ordered monomials of $\{ \breve{b}_{\alpha,r},\breve{h}_{\alpha,s}\}$ form a spanning set of $\breve{{}'\Y} $ and hence the ordered monomials of $\{b_{\alpha,r},h_{i,s}\}$ form a spanning set of ${}'\Y$.

Following the discussion in Section~\ref{sec:G2rel} and Definition~\ref{def:G2}, we have an algebra homomorphism 
\begin{align*}
\Phi: {}'\Y \longrightarrow \Gr_{\bK}\tUi, 
\qquad  \X_{i,r}  \mapsto \ov{\X}_{i,r,1},\quad
\H_{i,s} \mapsto \ov{\H}_{i,s,1}, 
\end{align*}
where $ \X_{i,r,1},\H_{i,s,1}$ are elements in $\tUi$ defined in the same way as \eqref{def:HX} and  $\ov{\X}_{i,r,1},\ov{\H}_{i,s,1}$ are their images in $\Gr_{\bK}\tUi$. Using arguments similar to the proof of Theorem~\ref{thm:qiso}, the $\Phi$-images for ordered monomials in $\{ b_{\alpha,r},h_{\alpha,s}\}$ (with respect to a fixed total ordering) are linearly independent in $\Gr_{\bK}\tUi$, and hence the ordered monomials in $\{ b_{\alpha,r},h_{\alpha,s}\}$ are linearly independent in ${}'\Y $.

Therefore, the ordered monomials in $\{b_{\alpha,r},h_{i,s}\}$ in ${}'\Y$ form a basis of ${}'\Y$.
\end{proof}

\begin{remark}
    By comparing the PBW basis for ${}'\Y$ in Theorem~\ref{thm:PBWbasisG2} and the corresponding PBW basis for $U\big(\g[u]^{\check\omega}\big)$, we see that ${}'\Y$ is a flat $\N$-graded deformation of $U\big(\g[u]^{\check\omega}\big)$ over $\C[\h]$.
\end{remark}

	\vspace{3mm}
\noindent{\bf Funding and Competing Interests.}	

We thank Curtis Wendlandt for helpful discussions. We thank a referee for very helpful suggestions which helped improving the presentation of this paper. KL and WW are partially supported by DMS--2401351. WZ is partially supported by  the New Cornerstone Foundation through the
New Cornerstone Investigator grant, and by Hong Kong RGC grant 14300021, both awarded to Xuhua He.

The authors have no competing interests to declare that are relevant to the content of this article.
The authors declare that the data supporting the findings of this study are available within the paper.

\end{document}